\documentclass[11pt,A4paper]{article}

\usepackage[left=25mm,right=25mm,top=25mm,bottom=30mm]{geometry}
\usepackage{epsfig}
\usepackage{epstopdf}
\usepackage{float}
\usepackage{thm-restate}
\usepackage{xfrac}

\usepackage[percent]{overpic}

\usepackage{relsize}
\usepackage{xcolor}
\usepackage{color}
\usepackage{amsthm,amsmath,amssymb}
\usepackage{booktabs}
\usepackage{mathpazo}
\usepackage{microtype}
\usepackage[T1]{fontenc}
\usepackage{bm}
\usepackage{sectsty}
\usepackage[
	pdftitle={PDFTitle},
	pdfauthor={Arnaud de Mesmay, Niloufar Fuladi, Hugo Parlier},
	ocgcolorlinks,
	linkcolor=linkred,
	citecolor=linkred,
	urlcolor=linkblue]
{hyperref}
\usepackage{cleveref}
\crefname{lemma}{Lemma}{Lemmas}

\usepackage{mathtools}

\definecolor{linkred}{RGB}{199,21,133}

\definecolor{linkblue}{RGB}{75,0,130}

\usepackage[hang,flushmargin]{footmisc}
\usepackage{enumitem}
\usepackage{titlesec}
	\titlespacing{\section}{0pt}{12pt}{0pt}
	\titlespacing{\subsection}{0pt}{6pt}{0pt}

\titlelabel{\thetitle.\quad}

\makeatletter 

\long\def\@footnotetext#1{%
\H@@footnotetext{%
\ifHy@nesting 
\hyper@@anchor{\@currentHref}{#1}%
\else 
\Hy@raisedlink{\hyper@@anchor{\@currentHref}{\relax}}#1%
\fi 
}}

\def\@footnotemark{%
\leavevmode 
\ifhmode\edef\@x@sf{\the\spacefactor}\nobreak\fi 
\H@refstepcounter{Hfootnote}%
\hyper@makecurrent{Hfootnote}%
\hyper@linkstart{link}{\@currentHref}%
\@makefnmark 
\hyper@linkend 
\ifhmode\spacefactor\@x@sf\fi 
\relax 
}%

\ifFN@multiplefootnote%
\renewcommand*\@footnotemark{%
\leavevmode 
\ifhmode 
\edef\@x@sf{\the\spacefactor}%
\FN@mf@check 
\nobreak 
\fi 
\H@refstepcounter{Hfootnote}%
\hyper@makecurrent{Hfootnote}%
\hyper@linkstart{link}{\@currentHref}%
\@makefnmark 
\hyper@linkend 
\ifFN@pp@towrite 
\FN@pp@writetemp 
\FN@pp@towritefalse 
\fi 
\FN@mf@prepare 
\ifhmode\spacefactor\@x@sf\fi 
\relax%
}%
\fi 

\makeatother 

\theoremstyle{plain}
\newtheorem{theorem}{Theorem}[section]

\newtheorem{lemma}[theorem]{Lemma}

\makeatletter
\newtheorem*{rep@theorem}{\rep@title}
\newcommand{\newreptheorem}[2]{%
\newenvironment{rep#1}[1]{%
 \def\rep@title{#2 \ref{##1}}%
 \begin{rep@theorem}}%
 {\end{rep@theorem}}}
\makeatother

\newreptheorem{theorem}{Theorem}
\newreptheorem{corollary}{Corollary}

\theoremstyle{definition}

\newtheorem{question}[theorem]{Question}

\newtheorem{remark}[theorem]{Remark}

\newcommand{\G}{{\mathcal G}}
\newcommand{\T}{{\mathcal T}}

\newcommand{\Sum}{\mathlarger{\mathlarger{ \sum}}}

\sectionfont{\large \bfseries}
\subsectionfont{\normalsize}

\long\def\symbolfootnote[#1]#2{\begingroup%
\def\thefootnote{\fnsymbol{footnote}}\footnote[#1]{#2}\endgroup}

\def\blfootnote{\xdef\@thefnmark{}\@footnotetext}

 \def\DEBUG{truetrue}
\ifdefined\DEBUG

\begin{document}

{\Large \bfseries 
Universal families of arcs and curves on surfaces}

{\large Niloufar Fuladi, Arnaud de Mesmay and Hugo Parlier\symbolfootnote[1]{All authors were partially supported by grant ANR-17-CE40-0033 of the French National Research Agency ANR (project SoS) and INTER/ANR/16/11554412/SoS of the Luxembourg National Research fund FNR: https://SoS.loria.fr/\\
{\em 2020 Mathematics Subject Classification:} Primary: 57K20 Secondary: 32G15, 57M15\\
{\em Key words and phrases:} curves and arcs on surfaces, triangulations, pants decompositions, mapping class groups}

{\bf Abstract.} 
The main goal of this paper is to investigate the minimal size of families of curves on surfaces with the following property: a family of simple closed curves $\Gamma$ on a surface \emph{realizes all types of pants decompositions} if for any pants decomposition of the surface, there exists a homeomorphism sending it to a subset of the curves in $\Gamma$. The study of such universal families of curves is motivated by questions on graph embeddings, joint crossing numbers and finding an elusive center of moduli space. In the case of surfaces without punctures, we provide an exponential upper bound and a superlinear lower bound on the minimal size of a family of curves that realizes all types of  pants decompositions. We also provide upper and lower bounds in the case of surfaces with punctures which we can consider labelled or unlabelled, and investigate a similar concept of universality for triangulations of polygons, where we provide bounds which are tight up to logarithmic factors.

\section{Introduction}

The study of simple curves and arcs on surfaces has played an ubiquitous role in geometric topology and combinatorial geometry, bringing together topics such as the study of mapping class groups, Teichm\"uller spaces and graph drawings. In particular, they have played an important role in understanding combinatorial models for moduli-type spaces such as curve and arc graphs and their related cousins of pants graphs and flip-graphs for triangulations. In these contexts, curves and arcs are considered up to isotopy and hence self-homeomorphisms of surfaces act nicely on the space of isotopy classes. While, in general, the space of isotopy classes of curves or arcs is infinite, up to homeomorphism there are only finitely many topological types. More generally, the same phenomenon holds for families of multicurves - that is, collections of disjoint curves such as the set of all pants decompositions of a given surface. In this paper, we try and find sets of arcs and curves that allow one to construct all types of multicurves in a given family with minimal cardinality. 

In this context, curves and arcs are assumed to be simple. Our surfaces will be topological, orientable and of finite type, and hence are determined by their genus $g$ and number of punctures $n$. While they are allowed to be $0$, we require that the Euler characteristic ($=2-2g-n$) be negative. All such surfaces can be built by pasting pairs of pants (a surface homeomorphic to a sphere minus 3 points) and so pants play a fundamental role as building blocks in the study of surfaces, including for the study of geometric structures and related moduli spaces. Equivalently, given a finite type surface, by cutting along a sufficient number of disjoint curves (a pants decomposition), one obtains a collection of pants. For this reason, pants decompositions are among the most well-studied multicurves.

The homeomorphism type of a pants decomposition is entirely determined by the trivalent graph encoding the adjacencies of the different pants that it is made of. For example, there are two types of pants decompositions in genus $2$, which correspond to the two trivalent graph on two vertices (see Figure~\ref{F:exampleintro}). One of the main objects of study in this paper are families of curves which realize all topological types of pants decompositions. A set of curves $\Gamma$ is said to be a universal family (for pants decompositions) if for any pants decomposition of the surface, there exists a homeomorphism sending it to a subset of the curves in $\Gamma$. For example, in the genus $2$ case, there is a universal family of size $4$, pictured in Figure~\ref{F:exampleintro}.

\begin{figure}
\centering
\includegraphics[width=\textwidth-2cm]{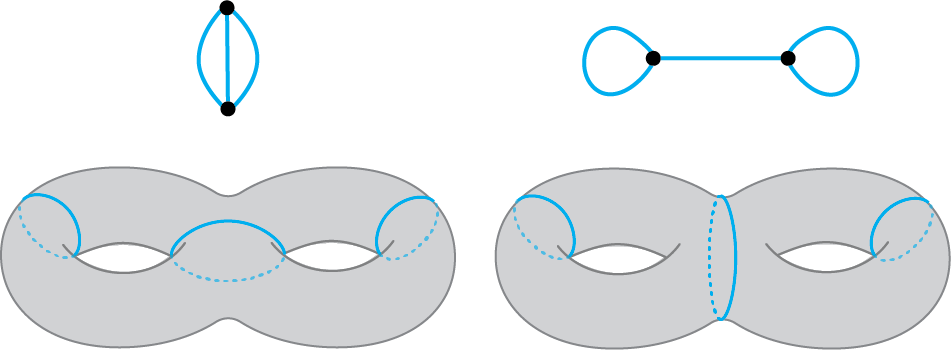}
\caption{The two homeomorphism classes of pants decomposition on a closed genus two surface, and their associated intersection graphs. The family of four curves drawn in the two pictures are enough to build both pants decompositions.}
\label{F:exampleintro}
\end{figure}

For surfaces of genus $g$, there are $g^{\Theta(g)}$ homeomorphism classes of pants decompositions~\cite{bollobas}, and thus, by taking an arbitrary pants decomposition in each homeomorphism class and the $3g-3$ curves it is made of, there is a trivial upper bound of $g^{O(g)}$ on the minimal size of a family of curves that realizes all types of pants decompositions. Our first result is to improve on this trivial bound to bring it to a singly-exponential dependency. 

\begin{theorem}\label{maintheorem}
Let $\Sigma$ be a closed orientable surface of genus $g$, and $\Gamma$ be a minimal size universal family for pants decompositions. Then

\[|\Gamma|\leq 3^{2g-1}\textrm{ and } |\Gamma|=\Omega(g^{4/3-\varepsilon})\]
for any $\varepsilon>0$.
\end{theorem}

The tantalizing gap between the exponential upper bound and the polynomial lower bound is the main open problem that we would like to advertise with this paper. 

The upper bound in Theorem~\ref{maintheorem} follows from an upper bound for the same problem on spheres with punctures. There, one can distinguish between the cases of labelled or unlabelled punctures, which radically changes the bounds: in the first case we consider homeomorphisms keeping the punctures fixed, while in the second case the punctures are allowed to be permuted. Our results are as follows.

\begin{theorem}\label{puncture}
\begin{itemize}
\item Let $\Sigma$ be a sphere with $n$ labelled punctures, and $\Gamma$ be a family of curves with minimal size that realizes all types of pants decompositions. Then

\[2^{n-1}-n-1\leq |\Gamma|\leq 3^{n-1}.\]

\item Let $\Sigma$ be a sphere with $n$ unlabelled punctures, and $\Gamma$ be a  family of curves with minimal size that realizes all types of pants decompositions. Then

\[|\Gamma|=O(n^2) \textrm{ and } |\Gamma|=\Omega(n\log n)\]
\end{itemize}
\end{theorem}

In section~\ref{S:smallgenus}, we provide universal families for surfaces of small genus with labelled punctures which suggest that the upper bound in Theorem~\ref{maintheorem} can be improved upon, but the approach seems unlikely to provide a subexponential bound.

Perhaps surprisingly (at least to us), the bound in the second item of Theorem~\ref{puncture} can be improved if instead of asking for families that realize all types of pants decompositions, we merely ask for a family of curves realizing all the types of pairs of pants $P \subseteq S_n$ (and not the whole decomposition). The homeomorphism class of such a pair of pants is entirely determined by the triplet $(k_1,k_2,k_3)$ counting the number of punctures in the three components it separates, where $k_1+k_2+k_3=n$. There are $O(n^2)$ such triplets, but we provide a random construction that achieves a better bound:

\begin{restatable}{theorem}{unlabel}\label{L:unlabelledpants}
There exists a family of simple closed curves of size $O(n^{4/3}\log^{2/3}n)$ on the sphere with $n$ punctures $S_n$ that realizes all types of pants.
\end{restatable}

Finally, one can phrase similar universality conditions for families of edges connecting punctures. For example, in the planar case, given a polygon with $n$ unlabelled vertices, one can look for families of edges realizing all the homeomorphism classes of triangulations of $\pi_n$, or merely all the homeomorphism classes of triangles in $\pi_n$. In that case, the situation is constrained enough that we can provide upper bounds and lower bounds which are almost tight.

\begin{theorem}\label{triangless}
In a polygon with $n$ vertices, the minimal size of a family of edges $E$ realizing all triangulations satisfies $|E|=O(n^2)$ and $|E|=\Omega(n^{2-\varepsilon})$ for any $\varepsilon>0$. The minimal size of a family of edges $E$ realizing all types of triangles satisfies $|E|=O(n^{4/3}\log^{2/3}n)$ and $|E|=\Omega(n^{4/3})$.
\end{theorem}

One can similarly investigate families of edges realizing all the one-vertex (or several-vertex) triangulations of surfaces. The techniques that we use for pants decompositions apply equally well for that setting. For the sake of avoiding redundancies, we do not include the corresponding results in this paper.

{\bf Motivations.} An important motivation for the work in this paper is the following question raised by Hubard, Kalu\v{z}a, the second author and Tancer in~\cite{shortestpaths}:

\begin{question}\label{Q:shortestpaths}
Given a surface $\Sigma$ of genus $g$, does there exist a Riemannian metric on $\Sigma$ such that any graph embeddable on $\Sigma$ can be embeddded so that the edges are shortest paths on $\Sigma$?
\end{question}

Such a metric would provide a strong generalization of the well-known F\`ary-Wagner~\cite{fary,wagner} theorem stipulating that any planar graph can be embedded in the plane so that the edges are straight lines. This question is furthermore motivated by problems around the joint crossing numbers of graphs on surfaces, and is known to have an affirmative answer if we merely ask for geodesics instead of shortest paths, or if we allow the metric to depend on the graph (using circle packing)---we refer to~\cite{shortestpaths} for more background and known results around this question.

The connection between Question~\ref{Q:shortestpaths} and the problems studied in this paper can be easily explained as follows. Given a pants decomposition of a surface of genus $g$, it is easy (see for example~\cite[Section~5]{shortestpaths}) to subdivide each curve a constant number of times and connect the resulting vertices with $O(g)$ edges so that we obtain a connected embedded graph. Furthermore, one can ensure that this graph has the property that if it is embedded so that the edges are shortest paths, then the original pants decomposition can be realized so that each curve consists of a constant number of shortest paths. Therefore, if the answer to Question~\ref{Q:shortestpaths} is affirmative, it means that there exists a metric on $\Sigma$ so that \emph{any} pants decomposition can be realized so that each curve consists of a constant number of shortest paths. Furthermore, by a standard cut and paste argument, shortest paths pairwise cross at most once. Therefore, an affirmative answer to Question~\ref{Q:shortestpaths} would imply that there exists a family of curves realizing all types of pants decompositions on $\Sigma$ so that any pair of curves crosses a constant number of times. But then, such a family of curves would need to have polynomial size~\cite{przytycki,greene} in $g$. Thus if we could prove that any family realizing all types of pants decompositions must have superpolynomial size, this would provide a negative answer to Question~\ref{Q:shortestpaths}. Our result in Theorem~\ref{maintheorem} provides partial results in this direction. 

In a completely different setting, finding universal families is of interest in Teichm\"uller theory, and in particular in the study of the Weil-Petersson metric. Without going into too many details, this metric on Teichm\"uller space, which is a deformation space of marked hyperbolic surfaces of given topology, has the special property that one can "pinch" a family of disjoint simple curves until they are length $0$ in finite time. In particular, one can find a metric completion of the space by adding degenerate surfaces with pinched curves (so-called noded surfaces). A hyperbolic surface is determined by lengths and twists of curves in a pants decomposition and hence a surface with pinched pants is fully determined by the combinatorics of its decomposition. Now, applying the group of self-homeomorphisms (the mapping class group) to the metric completion of Weil-Petersson space sends these fully degenerate surfaces to the set of types of pants decompositions. The quotient of Teichm\"uller space under the action of the mapping class group is moduli space, and with the induced metric this is Weil-Petersson moduli space, a well-studied object currently used to great effect in geometric probability theory (see \cite{guth2011pants,mirzakhani}). The geometry of this space - and in particular the asymptotic geometry as the genus grows - has been studied from different angles, and certain geometric invariants, such as the inradius \cite{wu2019growth} and the diameter \cite{10.1215/00127094-1507312} are somewhat well understood. There would be different possibilities to define a {\it center} for this moduli space - a type of midpoint that is as close as possible to all other points - but the Weil-Petersson metric is tricky to work with, and distances are difficult to compute. In light of what precedes, taking a type of barycentric average over the degenerate surfaces defined above, would be a reasonable choice for, at least, a rough centerpoint. By a theorem of Wolpert, the distance between a surface $X$ and a fully degenerate surface is bounded above by $\sqrt{2\pi L}$ where $L$ is the length of the pants decomposition on $X$ that one pinches. And so a surface which "wears" all topological types of pants decompositions well (that is such that their length is relatively short) would be a good candidate for being a center point. Finding a candidate surface would be a type of dual problem to the problem of finding a surface whose shortest pants decomposition is of maximal length (the study of the so-called Bers' constant, see for instance \cite{buser}). And this requires controlling the lengths of curves that you make the pants decompositions out of, and the fewer lengths you need to control, the easier this is. And hence, in this context as well, the question of how many curves you need to fabricate all types of pants decompositions is relevant. If the number of necessary curves is large (for instance close to our upper bounds) then by a counting argument, this would provide a lower bound for the length of curves in the center of moduli space. If, on the contrary, fewer curves are necessary, this would be helpful when constructing a candidate center as one only has to try and minimize the length of curves in a smaller set. 

Finally, let us mention that there are various related problems in combinatorics and graph theory where one investigates a given class of combinatorial objects and aims at finding a universal family that contains the entire class in some way. Such investigations date back at least to Rado~\cite{rado}. Depending on the precise notions considered, the size of this universal object might or might not be exponentially larger than the size of the objects in the class. A basic example of this is a \emph{k-universal permutation}, or \emph{k-superpattern} on $n$ symbols, which is a permutation containing all the possible patterns on $k$ symbols as subsequences. The smallest known k-superpatterns have size quadratic in $k$, but the best possible constant is still unknown, see for example Miller~\cite{miller}. In contrast, one might consider \emph{superpermutations}, which are strings on $n$ letters where all the permutations of size $n$ appear as substrings: one can easily prove that such superpermutations necessarily have a size exponentially large in $n$. Many universality notions for graphs have been investigated, a famous problem being the optimal size of a universal graphs containing all graphs of a fixed size as an induced subgraph, see Alon~\cite{alon} and the references therein. Perhaps closest to this work is a series of recent papers~\cite{bonamy,dujmovic,esperet} on universal graphs that contain all planar or bounded genus graphs of a fixed size as a subgraph or an induced subgraph: such universal graphs only have polynomial size (actually near-linear), but the problems that we study in this paper do not seem to be amenable to these techniques.

{\bf Structure.} This article is structured as follows. After a short preliminary section, we treat the case of unlabelled spheres and polygons. In Section \ref{sec:labpunctures}, we investigate spheres with labelled punctures. Results on the asymptotic growth in terms of genus for closed surfaces of universal families are in Section \ref{sec:genus}. In the final section, we show how to adapt the results for punctured spheres to surfaces of small genus. While this last part is technical, with somewhat incremental progress, it illustrates how adding genus increases the complexity of the problem, shedding light on the size of the gap in Theorem \ref{maintheorem}.

\section{Preliminaries and notations}

Throughout this paper, we denote by $\Sigma$ an orientable surface of genus $g$ with $n$ punctures (where $n$ can be zero) of negative Euler characteristic, i.e., $2g+n>2$. A \emph{pair of pants} (or just \emph{pants}) is a topological surface homeomorphic to a sphere with $3$ punctures, and a \emph{pants decomposition} is a set of simple and disjoint curves cutting a surface into a family of pairs of pants. Pairs of pants are the simplest possible surfaces that one can obtain after cutting along simple closed curves, and thus constitute fundamental building blocks in the study of surfaces, their geometric structures and moduli spaces (see for example~\cite{farb2011primer}).

In this article, we will be working with different surfaces and some of their substructures (e.g., triangulations, pairs of pants and pants decompositions). We denote the $n$-sided polygon with $\pi_n$, which we consider up to homeomorphism which are allowed to rotate the punctures. A \emph{triangle} on $\pi_n$ is the homeomorphism class of a triangle with endpoints on distinct vertices of $\pi_n$. A \emph{triangulation} of $\pi_n$ is 
a maximal set of interior disjoint triangles in $\pi_n$.

We denote the surface of genus $g$ and $n$ punctures by $\Sigma_{g,n}$, using $\Sigma_g:=\Sigma_{g,0}$ and $S_n:=\Sigma_{0,n}$ for shorthand. We consider two kinds of homeomorphisms for surfaces: homeomorphisms that fix punctures pointwise or globally (i.e., that can permute them). We will refer to the first case as the \emph{labelled} case and the second case as the \emph{unlabelled} case.

\section{Unlabelled punctures}
\subsection{Unlabelled punctures: Realizing pants and triangles}
In this section, we only consider planar surfaces, i.e., the polygon $\pi_n$ and the surface $S_n$. We number the vertices/punctures (in consecutive order for the polygon) from $1$ to $n$. We are working in the unlabelled setting, where the punctures are indistinguishable. Therefore, the \emph{topological type} of a pair of pants $P \subseteq \Sigma_{0,n}$ is the triple $(k_1,k_2,k_3)$ such that the three components of $S_n \setminus P$ have $k_1, k_2$ and $k_3$ punctures (we do not count the puncture $\partial P$). Without loss of generality, we always assume that $k_1\leq k_2 \leq k_3$, and note that in a topological type we always have $k_1+k_2+k_3=n$. We say that a family of simple closed curves $\Gamma$ \emph{realizes all types of pants} if for any topological type, there exist three disjoint curves in $\Gamma$ bounding a pair of pants that realizes that type.

\subsubsection{Upper bounds}

\unlabel*

\begin{proof}
In this proof, we work with arithmetic modulo $n$ and the interval notations are also modulo $n$: for example $[|n-2,2|]=\{n-2,n-1,n,1,2\}$.

The proof is based on a random construction. Let us denote by $\gamma_{i,j}$ a simple closed curve separating all the punctures in $[|i,j-1|]$ from the others, and such that $\gamma_{i,j}$ and $\gamma_{k,\ell}$ are disjoint whenever $[|i,j-1|]$ and $[|k,\ell-1|]$ are disjoint, see Figure~\ref{F:standardcurves}. If $i=j$ we take a simple contractible curve disjoint from all the others.

\begin{figure}
\centering
\includegraphics[width=11cm]{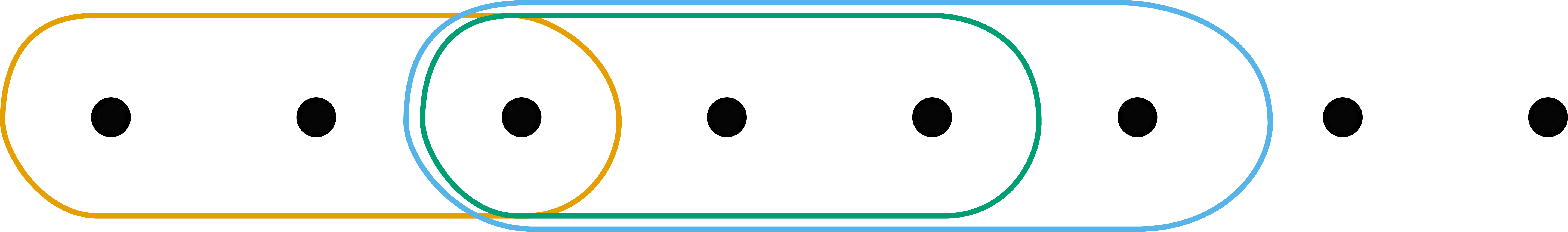}
\caption{A sphere with eight punctures with a choice of curves $\gamma_{1,3}$, $\gamma_{3,5}$ and $\gamma_{3,6}$.}
\label{F:standardcurves}
\end{figure}

We set $p=c\log^{1/3}n/n^{1/3}$ and define a random set $S \subseteq [|1,n|]$ by putting each integer in $[|1,n|]$ in $S$ with probability $p$. Then by Chernoff bounds, the set $S$ has size at most $2cn^{2/3}\log^{1/3}n$ with probability at least $1-e^{-cn^{2/3}\log^{1/3}n/3}>1/2$ for big enough $c$. We define $\Gamma:=\{\gamma_{i,j} \mid (i,j) \in S^2\}$. Now $\Gamma$ has size at most $4c^2n^{4/3} \log^{2/3} n$ with probability at least $1/2$. 

We now show that there is a nonzero probability that $\Gamma$ realizes all types of pants. Let $(k_1,k_2,k_3)$ be a topological type (thus we have $k_1+k_2+k_3=n$). For $i,j,k$ three integers in $[|1,n|]$, then the three curves $\gamma_{i,j}$, $\gamma_{j,k}$ and $\gamma_{k,i}$ bound a pair of pants $P_{i,j,k}$, which is of type $(k_1,k_2,k_3)$ if and only if $j-i=k_1$ and $k-j=k_2$. So for $i \in [|1,n|]$, we denote by $X^{k_1,k_2,k_3}_i$ the random variable indicating the event that $(i,i+k_1,i+k_1+k_2)$ belongs to $S$, which happens with probability at least $p^3$ (note that if $k_1=0$ or $k_2=0$, the probability is higher). Then by Chernoff bounds, the probability that $(k_1,k_2,k_3)$ is not realized is equal to $P(\sum_i X_i^{k_1,k_2,k_3}=0)\leq e^{\frac{-p^3}{2n}}\leq \frac{1}{4n^2}$ for big enough $c$. Note that this probability does not depend on $(k_1,k_2,k_3)$, and therefore by the union bound there is a probability at least $1/4$ that all types of pants are realized. Since $1/4+1/2<1$, with nonzero probability we have the correct bound on the size of $\Gamma$ and it realizes all types of pants, concluding the proof.
\end{proof}

Similarly to the problem of realizing all types of pants, we can look for families of edges realizing triangles in a polygon $\pi_n$, where the goal is to realize all types of triangles $T$. A type of triangle is determined by a triple $(k_1,k_2,k_3)$ such that the three components of $\pi_n \setminus T$ contain respectively $k_1, k_2$ and $k_3$ vertices of the polygon $\pi_n$. Now, the exact same proof provides the following, which mirrors Lemma~\ref{L:unlabelledpants}.

\begin{lemma}\label{L:unlabelledtriangles}
There exists a family of edges on $\pi_n$ of size $O(n^{4/3} \log^{2/3} n)$ realizing all types of triangles.
\end{lemma}

\subsubsection{Lower bounds}

The bound obtained in Lemma~\ref{L:unlabelledtriangles} is sharp up to logarithmic factors.

\begin{lemma}\label{L:lbtriangles}
For a polygon $\pi_n$, any family realizing all types of triangles has size at least $\Omega(n^{4/3})$.
\end{lemma}
\begin{proof}
The vertices of the polygon $\pi_n$ and the edges in a family $E$ realizing all types of triangles form a graph. It is known (see for example Rivin~\cite{rivin}) that any graph with $|E|$ edges has $O(|E|^{3/2})$ triangles (i.e., cycles of length $3$). Therefore, in order to realize  $\Theta(n^2)$ types of triangles, the graph must have at least $\Omega(n^{4/3})$ edges.
\end{proof}

\begin{remark}
While we believe the bound for types of pants in Lemma~\ref{L:unlabelledpants} to be roughly sharp, as in Lemma~\ref{L:unlabelledtriangles}, the same argument does not apply. The only lower bound we know for the number of curves needed to realize all topological types of pants on $S_n$ with unlabelled punctures is the trivial lower bound of $\Omega(n)$.
\end{remark}
\subsection{Unlabelled punctures: Realizing pants decompositions and triangulations}\label{3.2}

We now turn our attention to pants decompositions and triangulations. The \emph{topological type} of a pants decomposition or triangulation is its homeomorphism type, and in the unlabelled case, it is completely determined by the trivalent tree encoding the adjacencies of the pairs of pants/triangles. As before, we say that a family of curves $\Gamma$ \emph{realizes all types of pants decompositions} if for any topological type of pants decomposition, there exists a subset of $3g-3$ curves in $\Gamma$ inducing that topological type. Therefore, it is trivial to realize all types of pants decompositions in $S_n$ using $O(n^2)$ curves, respectively edges: one can simply number the punctures from $1$ to $n$ arbitrarily and for any pair $(i,j)$, take a curve running around the punctures from $i$ to $j$, as pictured in Figure~\ref{F:standardcurves}. Likewise, one can realize all types of triangulations in $\pi_n$ using $O(n^2)$ edges. For polygons, we can prove an almost matching lower bound.

\begin{lemma}\label{L:lbtriangulations}
For any $\varepsilon>0$, any family of edges realizing all triangulations of $\pi_n$ has $\Omega(n^{2-\varepsilon})$ edges.
\end{lemma}

\begin{proof}
The proof follows the same idea as the proof of Lemma~\ref{L:lbtriangles}. Any family of edges realizing all triangulations defines a graph with $|E|$ edges. Since the family of edges realizes all types of triangulations, in particular it realizes all types of triangles, $4$-cycles, or more generally $\ell$-cycles. The type of an $\ell$-cycle $C$ is determined by the tuple $(k_1, \ldots, k_\ell)$ of vertices in each of the connected components of $\pi_n \setminus C$, and thus there are $\Theta(n^{\ell-1})$ of them. Now, any graph with $|E|$ edges has at most $O(|E|^{\ell/2})$ $\ell$-cycles, and thus $|E|=\Omega(n^{\frac{2\ell-1}{\ell}})$. The result follows by taking $\ell$ arbitrarily big.
\end{proof}

 Lemmas~\ref{L:unlabelledtriangles}, \ref{L:lbtriangles} and \ref{L:lbtriangulations} prove Theorem~\ref{triangless}.

Here again, this proof technique fails for pants decompositions. The best lower bound we can provide is the following.

\begin{lemma}\label{L:lbpantsdec}
Any family of curves realizing all types of pants decompositions of $S_n$ has size at least $\Sum_{i=2}^{\lfloor\frac{n}{2}\rfloor}\lfloor\frac{n}{i}\rfloor=\Omega(n \log n)$.
\end{lemma}
\begin{proof}
Let $T$ be a trivalent tree with $n$ leaves associated to a pants decomposition $P$ of $S_n$, in which each internal vertex corresponds to a pair of pants and each edge corresponds to a closed curve in $P$. For each curve $\gamma$ in $P$, we say that $\gamma$ \textbf{bounds} $i\leq \lfloor\frac{n}{2}\rfloor$ punctures if its corresponding edge in $T$ separates a sub-tree that has exactly $i$ leaves. We prove the following claim.

\textbf{Claim.} For each $i\leq\lfloor\frac{n}{2}\rfloor$, there exists a trivalent tree $T_i$ with $n$ leaves such that its corresponding pants decomposition contains $\lfloor\frac{n}{i}\rfloor$ curves, each of which bounds $i$ punctures.

Assuming the claim for now, for each $i\leq\lfloor\frac{n}{2}\rfloor$, the existence of such a tree implies that to realize the pants decomposition corresponding to this tree,   $\lfloor\frac{n}{i}\rfloor$ closed curves that bound $i$ punctures are needed. This implies that at least  $\Sum_{i=2}^{\lfloor\frac{n}{2}\rfloor}\lfloor\frac{n}{i}\rfloor$ closed curves are needed to realize $T_i$s for $2\leq i\leq \lfloor\frac{n}{2}\rfloor$ and proves the lemma.

\emph{Proof of the claim.} 
In a rooted tree, we say that a vertex is internal if it is not the root and not a leaf. Let $B_i$ be any trivalent tree with $\lceil\frac{n}{i}\rceil$ leaves and let $L_i$ be a rooted tree with $i$ leaves in which the root has degree 2 and the internal vertices have degree 3. Let $R_i$ be a rooted tree with $n-i\lfloor\frac{n}{i}\rfloor$ leaves in which the root has degree 2 and the internal vertices all have degree 3. We define a tree $T_i$ as follows: in the case where $i$ divides $n$, we paste a copy of $L_i$ on each leaf of $B_i$ by identifying the root in $L_i$ with a leaf in $B_i$. If $i$ does not divide $n$, we paste a copy of $L_i$ on every leaf of $B_i$ except one, and we paste $R_i$ on the last leaf. We can see that $T_i$ is a trivalent tree with $n$ leaves. Each edge in $T_i$ that corresponds to an edge that used to connect a leaf in $B_i$ separates $i$ leaves in $T_i$ and therefore its corresponding curve in the pants decomposition bounds $i$ punctures. We refer to Figure~\ref{jjj} for an illustration.
\begin{figure}[ht]
    \centering
    \includegraphics[width=0.73\textwidth]{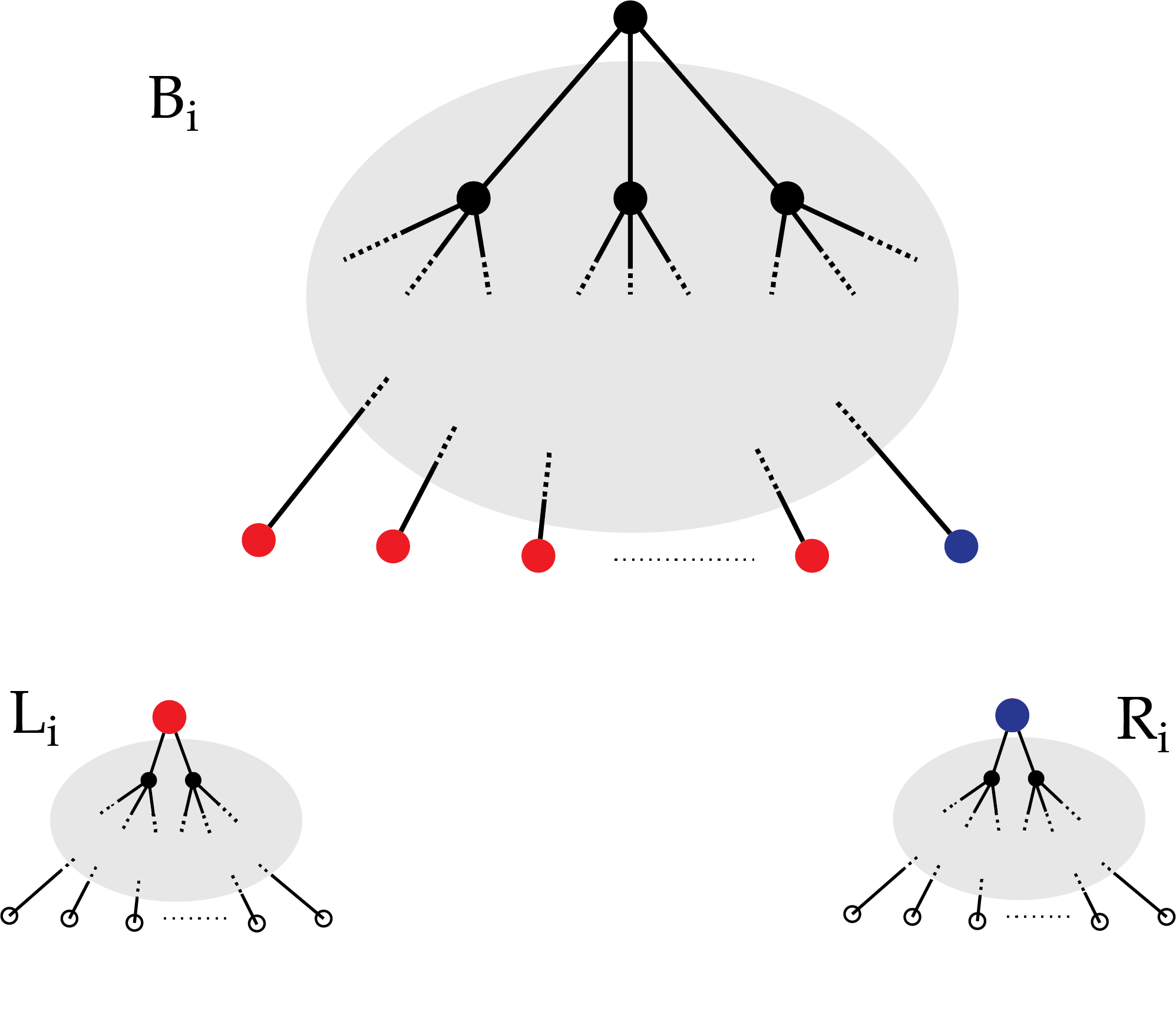}
    \caption{The case where $i$ does not divide $n$. $B_i$ has $\lceil\frac{n}{i}\rceil$ leaves, $L_i$ has $i$ leaves and $R_i$ has $n-i\cdot\lfloor \frac{n}{i}\rfloor $ leaves. 
}
    \label{jjj}
\end{figure}
\end{proof}

The construction at the beginning of Section~\ref{3.2} and Lemma~\ref{L:lbpantsdec} proves the second item of the Theorem~\ref{puncture}.

\section{Labelled punctures}\label{sec:labpunctures}

In this section, we turn our attention towards families of curves realizing all types of pants decompositions for spheres with labelled punctures. The dual graph to a pants decomposition $P$ of $S_n$ with labelled punctures is a trivalent tree $T$ with $n$ labelled leaves. We call an edge in a tree $T$ an internal edge if it is not adjacent to a leaf. The following lemma provides a family of curves of exponential size realizing all types of pants decompositions in this setting.

\begin{lemma}\label{lem:labelled sphere}
Let $S_n$ be the sphere with $n$ labelled punctures. There exists a family of simple closed curves of size less than $3^{n-1}$ that realizes all types of pants decompositions of $S_n$ up to labelled homeomorphism.
\end{lemma}

\begin{proof}

Let $P$ be a pants decomposition of $S_n$ and $T$ be the tree dual to $P$. Each internal edge of the tree separates the leaves into two parts and it corresponds to a simple closed curve on $S_n$ that separates the corresponding punctures. Here, we deal with a bigger set of trees than those that are dual to pants decompositions. Denote by $\T_n$ the set of all trees with $n$ leaves and with every internal vertex having degree at most 3.
In this proof, we deal with free homotopy classes of curves. Two homotopy classes of curves are said to be disjoint if there exists a representative in each class such that these curves do not intersect each other.

We remove a point from the sphere and consider the punctures on the plane, which we number arbitrarily from $1$ to $n$ and consider lined up from $1$ to $n$. For a subset $S\subset [n]$, and for each map $f:[n] \setminus S \rightarrow \{above,below\}$, we define a homotopy class as follows: the curve $\gamma_S^f$ encompasses the punctures in $S$, and for each puncture $b$ not in $S$ but within $[\min S,\max S]$, it goes above, respectively below, $b$, depending on the value of $f(b)$. These homotopy classes are pictured in Figure~\ref{Fig:wiggle}, for $S=\{2,6\}$ and the four possible maps $f$ choosing above or below for the punctures labelled $3$, $4$ and $5$. Note that for punctures not in $S$ and not between the smallest and the largest element of $S$, there is no choice of "above" or "below" to be made, hence no need to define $f$. Then we denote by $\Gamma_S$ the union of all the curves $\gamma_S^f$ for all the possible maps $f$.

Let $T\in \T_n$ and fix a vertex $v$ to be the root of $T$. We say that a set $\Phi_T$ of homotopy classes of curves recognizes the tree $T$ with respect to $v$ if it satisfies the following properties. The homotopy classes in $\Phi_T$ are pairwise disjoint. If $e$ is an internal edge in $T$ that separates a sub-tree that does not contain the vertex $v$ and that has $S$ as leaves, there exists a homotopy class of curves in $\Phi_T$ that encompasses the punctures in $S$ (intuitively, a virtual leaf attached to the root should correspond to the point at infinity).
We say that a set of homotopy classes of curves recognizes a tree $T$ if for any vertex $v$ in $T$, it recognizes $T$ with respect to $v$.
Then, in order to realize the pants decomposition corresponding to $T$, it is enough to choose a curve from each homotopy class of curves in $\Phi_T$ such that these curves are pairwise disjoint.  

\begin{figure}[ht]
    \centering
    \includegraphics[width=\textwidth]{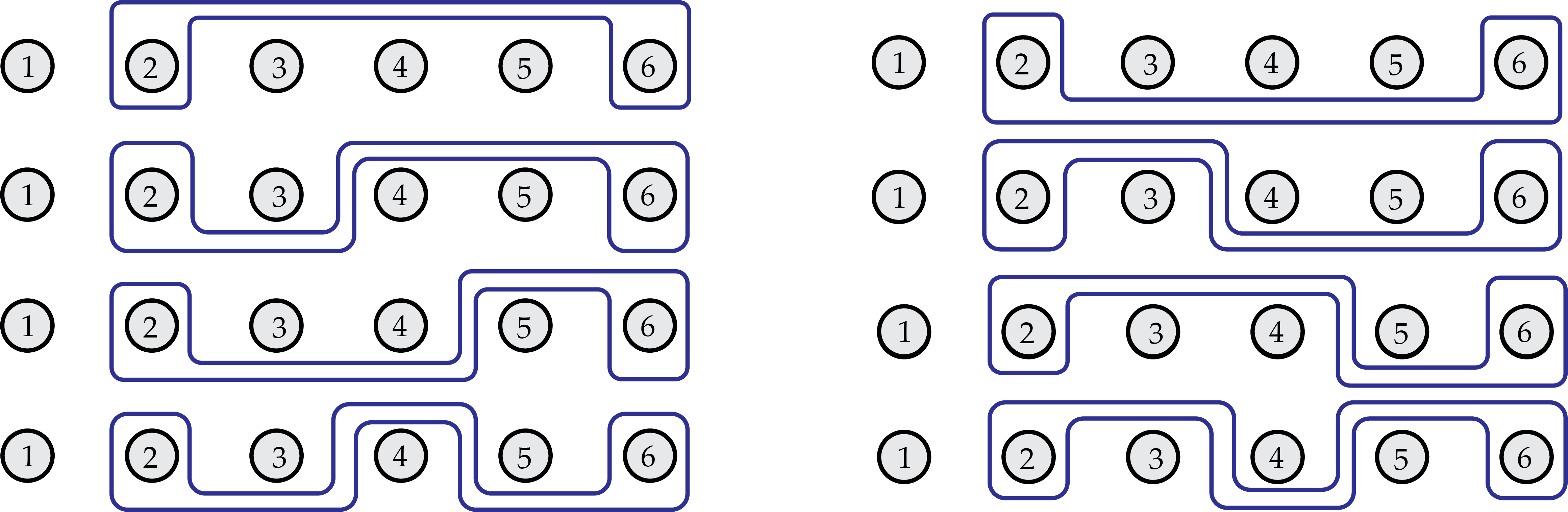}
    \caption{Eight possible choices for the curve that separates $\{2,6\}$ from $\{1,3,4,5\}$}
    
    \label{Fig:wiggle}
\end{figure}

Let $\Lambda_{[n]}=\bigcup_{S\subset [n], 1<|S|\leq n-2} \Gamma_S$. We first prove that each tree $T\in\T_n$ can be recognized by a set of homotopy classes in $\Lambda_{[n]}$ and then we show that  families of curves can be chosen from each homotopy class in $\Lambda_{[n]}$ such that they realize all types of pants decompositions of $S_n$.

We prove the following claim.

 \textbf{Claim.} The set of homotopy classes $\Lambda_{[n]}$ is enough to recognize all trees in $\T_n$.

 \emph{Proof of the claim:} The proof of the claim is by induction on $n$.
Pick a tree $T\in\T_n$ with a root $v$. 
Let $S_1,S_2$ and $S_3$ be the set of leaves that belong to each branch of $v$.
We need to recognize the edges adjacent to $v$ that are internal edges in the tree. 
If all the adjacent edges of $v$ are internal edges, then they correspond to three curves that separate the punctures in $S_1,S_2$ and $S_3$ from each other. For $S_1$ (resp. for $S_2$) choose the homotopy class of curves $\gamma_1$ (resp. $\gamma_2$) in $\Gamma_{S_1}$ (resp. $\Gamma_{S_2}$) that go above (resp. below) the punctures in $[n]\setminus S_1$ (resp. $[n]\setminus S_2$). For $S_3$, choose the homotopy class of curves $\gamma_3$ that go below the punctures in $S_1$ and above the punctures in $S_2$, see Figure~\ref{Fig:root}. The choice of above and below guarantees that these three homotopy classes are pairwise disjoint. 
\begin{figure}[H]
    \centering
    \includegraphics[width=.85\textwidth]{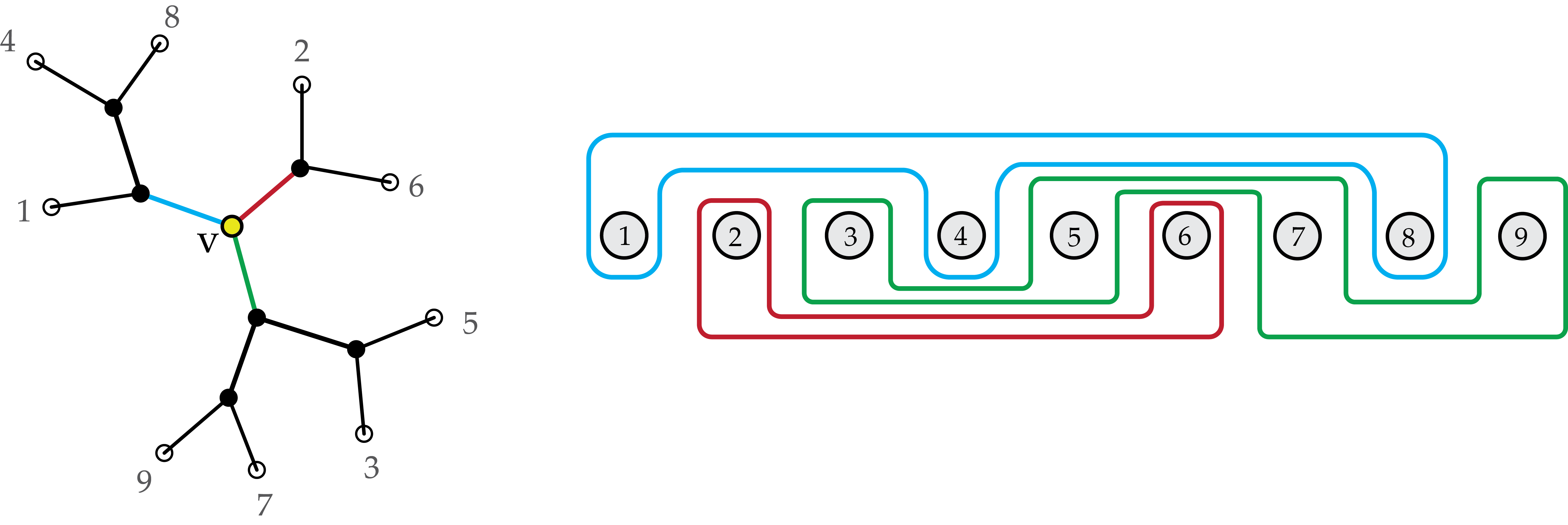}
    \caption{The curves realizing the edges adjacent to the root. Here $S_1=\{1,4,8\}$, $S_2=\{2,6\}$ and $S_3=\{3,5,7,9\}$.}
    \label{Fig:root}
\end{figure}

We denote by $T \setminus \{v\}$ the tree $T$ where we remove $v$ and the three adjacent edges. Let $T_i$ denote the sub-tree in $T\setminus \{v\}$ that corresponds to $S_i$. Let $v_i$ be the vertex in $T_i$ that used to be connected to $v$ and let it be the root of $T_i$; $v_i$ has degree 2. The number of leaves in $T_i$ for $1\leq i\leq 3$ is less than $n$ and therefore, by the induction hypothesis, we can recognize the edges in $T_i$ by a set $\Phi_{T_i}$ of pairwise disjoint
homotopy classes of curves in $\Lambda_{[S_i]}$ with respect to $v_i$. 
We can consider these homotopy classes as belonging to $\Lambda_{[n]}$ by making the homotopy classes in $\Phi_{T_i}$ go above or below the punctures in $S_i\setminus S$ in the same way that $\gamma_i$ does. By the way we chose the homotopy classes, they are pairwise disjoint and therefore they recognize the tree $T$. This finishes the proof of the claim.
 $\hfill\blacksquare$

Now, to realize the pants decomposition corresponding to a tree $T\in\T_n$, it suffices to choose a curve from each homotopy class in $\Phi_T$, such that these curves are pairwise disjoint. Note that by construction, these homotopy classes are pairwise disjoint. Thus, in order to realize all the types of pants decompositions, it suffices to pick a closed curve for each homotopy class in $\Lambda_{[n]}$ so that the resulting family is in minimal position (one way to do that is for example to fix an arbitrary hyperbolic metric on $S_n$ and pick geodesic representatives for each representative). 

We conclude the proof by upper bounding the size of the set $\Lambda_{[n]}$. Recall that for a set $S\subset [n]$, if $i\notin [\min(S),\max(S)]$ then the curves in $\Gamma_S$ do not need to go above or below the puncture $i$. Thus we have the following bound.
\begin{align*}
    |\Lambda_{[n]}|=\Sum_{S\subset [n], 1<|S|< n-1} |\Gamma_S|&=\Sum_{i=2}^{n-2}\quad \Sum_{k=0}^{n-i} (n-k-i+1)\binom{k+i-2}{i-2}2^k
    \\&=\Sum_{k=0}^{n-2} \quad\Sum_{i=k}^{n-2} (i-k+1) \binom{n+k-i-2}{k}2^k
    \\&=\Sum_{k=0}^{n-2}  2^k \left(\Sum_{j=1}^{n-k-1} j\binom{n-j-1}{k}\right)
    \\&=\Sum_{k=0}^{n-2} 2^k\binom{n}{k+2}
    \\&=\frac{1}{4}(3^n-2n-1)
\end{align*}
Note that in the second equation, for a subset $S$ of size $i$, $k$ counts the number of punctures not in $S$ that belong to $[\min(S),\max(S)]$ and $2^k$ counts the choices of going below or above these $k$ punctures. The term $(n-k-i+1)$ counts the number of cases for $[\min(S),\max(S)]$ in this case, i.e, the number of cases where $\max(S)-\min(S)= i+k-1$.
\end{proof}

The following lemma provides an easy exponential lower bound for this problem.

\begin{lemma}\label{lem:lowerbound}
Any family of curves realizing all types of pants decompositions of $S_n$ with labelled punctures has size at least $2^{n-1}-n-1$.
\end{lemma}
\begin{proof}
We know that for any subset $S$ of the punctures such that  $2\leq|S|\leq n-2$, there exists a closed curve in some pants decomposition that separates the punctures in $S$ from those in $[n]\setminus S$. Therefore we need at least $|\{S,[n]\setminus S, 2\leq|S|\leq n-2\}|$ closed curves to realize all types of pants decompositions of $S_n$ with labelled punctures, which is equal to $2^{n-1}-n-1$.
\end{proof}

Bridging the gap between the $3^{n-1}$ upper bound and the $2^{n-1}$ lower bound seems to be an interesting open problem as well.

 Lemmas~\ref{lem:labelled sphere} and \ref{lem:lowerbound} prove the first item of the Theorem~\ref{puncture}.

\begin{remark}
All the bounds in this section apply equally well to the problem of realizing all  types of \emph{pants} (instead of types pants decompositions), and both the upper and lower bounds are also the best ones we know for that problem.
\end{remark}

\section{Surfaces without punctures}\label{sec:genus}

An easy reduction to the case with labelled punctures, in conjunction with Lemma~\ref{lem:labelled sphere}, directly yields the following bound.

\begin{lemma}\label{lem:allpants}
On the surface $\Sigma_g$, there exists a family of curves of size at most $3^{2g-1}$ realizing all types of pants decompositions.
\end{lemma}

\begin{proof}
The family consists of taking $g$ simple disjoint closed curves such that cutting along them yields a sphere with $2g$ punctures. Then we use the family of curves given by Lemma~\ref{lem:labelled sphere}. This works because any pants decomposition contains $g$ simple disjoint closed curves such that cutting along them yields a sphere with $2g$ punctures. So we can realize these curves using our cutting curves. After cutting along these $g$ curves, the resulting pants decomposition is one on the sphere with $2g$ punctures, which can be labelled arbitrarily, and thus can be realized using a subset of the curves given by Lemma~\ref{lem:labelled sphere}. The labeling will in particular preserve the matching stipulating how to glue back the punctures, and thus our curves will indeed realize the targeted pants decomposition. 

Since the curves in Lemma~\ref{lem:labelled sphere} are of size at most $3^{2g-1}-g$ (we refer to the proof of that lemma which provides a bound that is slightly sharper than in the statement of the lemma), the resulting family has size at most $3^{2g-1}$.
\end{proof}

The technique in this proof can readily be adapted to provide a singly-exponentially sized family of curves realizing all types of pants decompositions on surfaces of genus $g$ with $n$ punctures.

The reduction in the proof of Lemma~\ref{lem:allpants} is clearly wasteful, in that we use a family of curves tailored to realize the pants decompositions for labelled punctures, but actually it would suffice to consider only pants decompositions up to homeomorphism which preserves the matching induced by the cutting curves. Due to this inefficiency, the exponential lower bound of Lemma~\ref{lem:lowerbound} does not translate to this setting, and we can only prove the following lower bound obtained with a counting argument.

\begin{lemma}\label{lem:lowerboundallpants}
For any $\varepsilon>0$, any family of simple closed curves realizing all types of pants decompositions on $\Sigma_g$ has size $\Omega(g^{4/3-\varepsilon})$.
\end{lemma}

\begin{proof}
The homeomorphism class of a pants decomposition is determined by the trivalent graph encoding the adjacencies of the pants. A pants decomposition has $2g-2$ curves, and there are, up to terms that are only exponential in $g$, $g^{g}$ such trivalent graphs on $2g-2$ vertices, and thus that many pants decompositions (see for example~\cite[Lemma~1]{guth2011pants}). Therefore, since any pants decomposition consists of $3g-3$ curves, any family of simple closed curves $\Gamma$ realizing them all must satisfy the counting lower bound given by $\binom{|\Gamma|}{3g-3}\gtrsim g^g$, where we use the $\gtrsim$ notation to hide terms that are only exponential in $g$. Therefore,

\[\left(\frac{|\Gamma| e}{3g-3}\right)^{3g-3}\geq \binom{|\Gamma|}{3g-3}\gtrsim g^g\]
and thus $|\Gamma|=\Omega(g^{4/3-\varepsilon})$ for any $\varepsilon>0$.
\end{proof}

Lemmas \ref{lem:allpants} and \ref{lem:lowerboundallpants} prove Theorem~\ref{maintheorem}.

\section{Small genus cases and labelled punctures: a small improvement}\label{S:smallgenus}

In this last section, we showcase that the approach of cutting along $g$ curves to planarize and then using the bound of labelled sphere is wasteful, as one can do better in the small genus cases when $g=1$ and $g=2$. While the proofs are somewhat ad hoc and will not generalize, we believe that this provides an interesting hint that the bound given in Lemma~\ref{lem:allpants} is not optimal.

First, let us consider the question of realizing all types of pants decompositions on $\Sigma_{1,n-2}$ when the punctures are labelled. Note that by cutting along a non-separating curve in $\Sigma_{1,n-2}$, we obtain a sphere with $n$ punctures and therefore by \cref{lem:labelled sphere}, a family of size at most $3^{n-1}$ is enough to realize all types of pants decompositions in this case. We improve this bound in the following lemma.

\begin{lemma}\label{lem:labelled torus}
There exists a family of simple closed curves of size less than $3^{n-2}$ that realizes all types of pants decompositions of $\Sigma_{1,n-2}$ up to labelled homeomorphisms.
\end{lemma}

\begin{proof}
The dual graph to a pants decomposition of $\Sigma_{1,n-2}$, is a graph that has only one cycle and in which all vertices have degree three, except $n-2$ vertices of degree one, which in our case, correspond to the labelled punctures and are labelled. Recall that we call an edge that is not adjacent to a labelled vertex an internal edge. Denote the set of all such graphs by $\G$. Note that the edges that belong to the cycle in $G$ are non-separating and correspond to curves that are non-separating, but any two such edges together separate $G$ into two sub-trees and therefore correspond to curves that are separating together. Recall that two free homotopy classes are disjoint if there exists a curve in each of them such that these curves do not intersect.

Let $\theta$ be a simple closed curve that is non-separating on $\Sigma_{1,n-2}$, as in Figure~\ref{Fig:nonseparating}. Denote by $\Sigma'$ the surface obtained by cutting along $\theta$: this is a sphere with $n$ punctures, among which $n-2$ are our labelled punctures. As in the proof of Lemma~\ref{lem:labelled sphere}, we think of these $n$ punctures as being lined up, and we consider, for each $S$ a subset of the labelled punctures and for each map $f:[n-2]\setminus S \rightarrow \{above,below\}$, a curve $\gamma^f_{S}$ which encompasses the punctures in $S$ and goes above or below the other punctures as specified by $S$. Gluing back the surface $\Sigma'$ along $\theta$, these curves together define a family of free homotopy classes $\Gamma_S$ on $\Sigma_{1,n-2}$.

Note that by construction, $\theta$ is disjoint from $\gamma_S^f$ for any $S\subset[n-2]$ and choice of $f$. Now, for each curve $\gamma_S^f$ where $f$ always maps to $above$, we want to define an alternate curve that additionally goes around the handle following $\theta$. Formally: using a path going above all of the punctures, we base all these curves at a point $b$ located somewhere on $\theta$ and consider the curves obtained by concatenating the curves $\gamma^f_S$ and $\theta$. This yields a second family of free homotopy classes which we denote by $\Theta_S$. We refer to Figure~\ref{Fig:theta} for the depiction of a curve $\gamma_S^f$ in red and its corresponding curve in $\Theta_S$ in blue.

\begin{figure}[t]
    \centering
    \includegraphics[width=.75\textwidth]{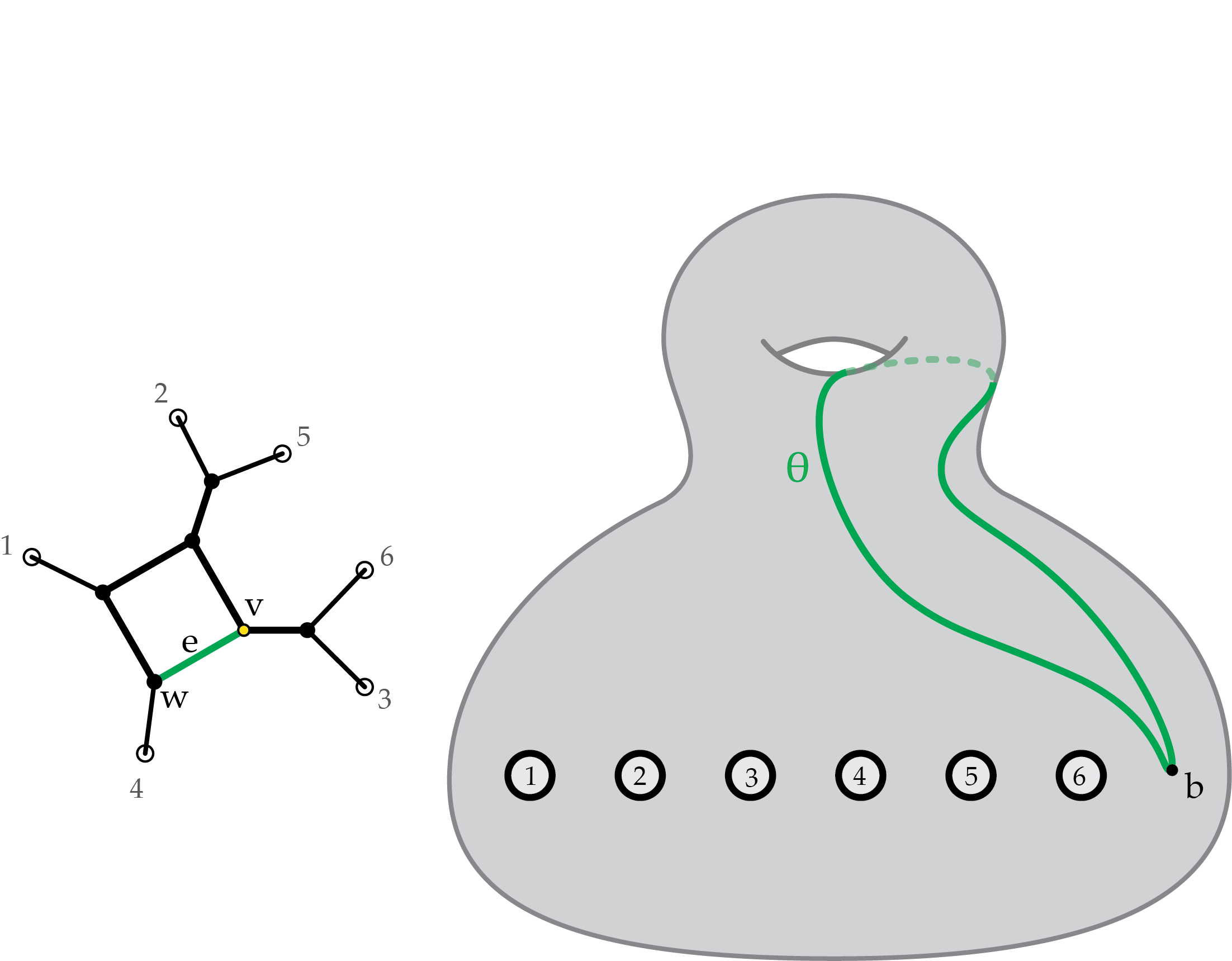}
    \caption{We realize the edge $e$ by the green curve ($\theta$) depicted at right.}
    \label{Fig:nonseparating}
\end{figure}

\begin{figure}[H]
    \centering
    \includegraphics[width=.4\textwidth]{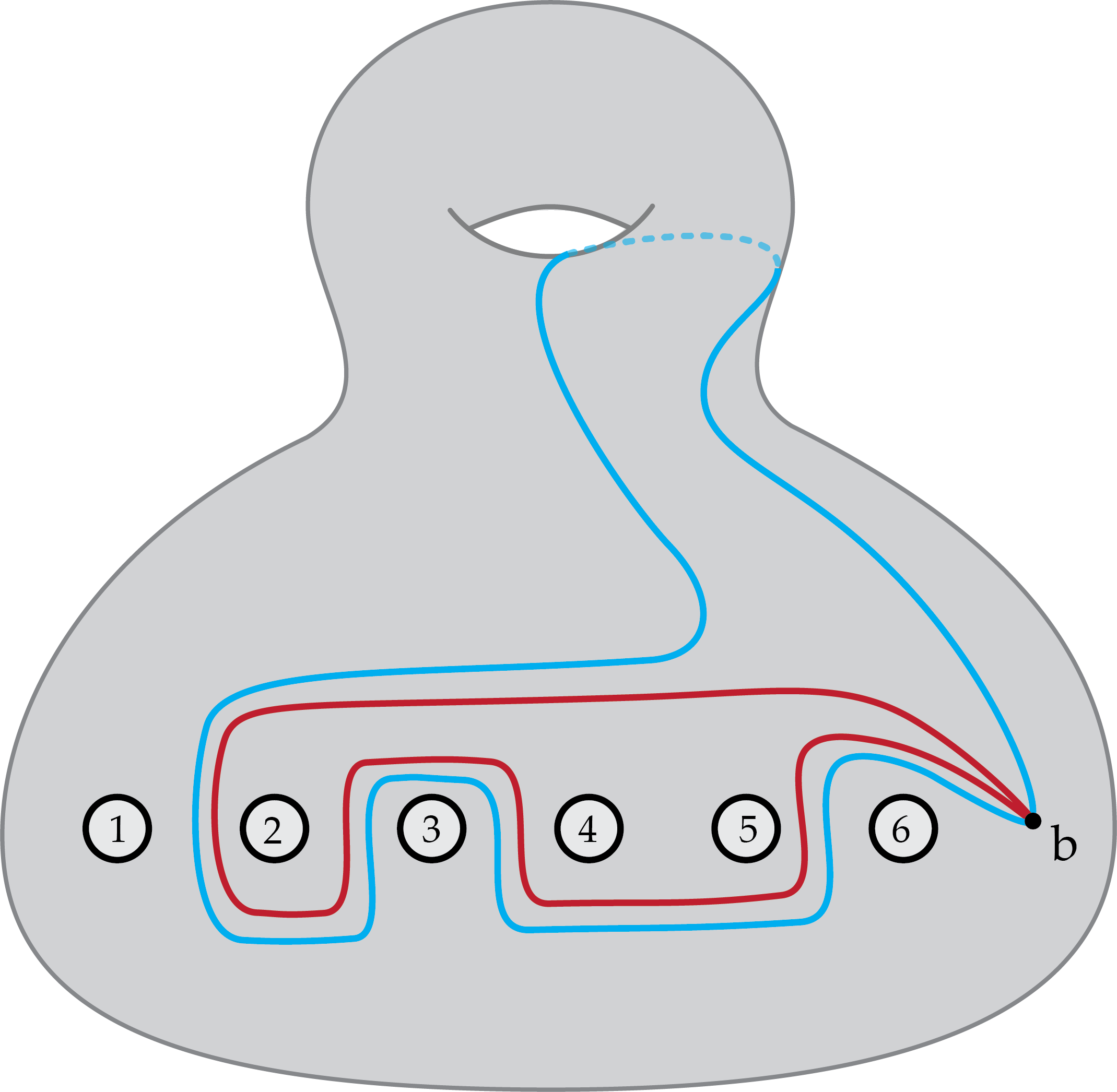}
    \caption{The red curve separates the surface to a sphere that has  $\{2,4,5\}$ as punctures and a genus one surface with $\{1,6\}$ as punctures; this curve belongs to $\Gamma_{\{2,4,5\}}$. The blue curve is non-separating and is the concatenation of the red curve with $\theta$, this curve belongs to $\Theta_{\{2,4,5\}}$.}
    \label{Fig:theta}
\end{figure}

We say that a set $\Phi_G$ of free homotopy classes of curves recognizes a graph $G\in \G$ if it satisfies the following properties.
\begin{itemize}
    \item The homotopy classes in $\Phi_G$ are pairwise disjoint.
\item For each internal edge $e$ in $G$ that is separating a tree that contains the subset $S$ of the labelled vertices, there exists a homotopy class of curves in $\Phi_G$ that cuts the surface into a sphere containing the punctures in $S$ and a genus one surface containing the punctures in $[n-2]\setminus S$.
\item Let $e$ and $e'$ be a pair of edges that separate a tree from $G$ containing the labelled vertices in $S$. 
Then there exist two homotopy classes of non-separating curves in $\Phi_G$ corresponding to $e$ and $e'$ that together separate the surface into a sphere that contains the punctures in $S$ and a sphere containing the punctures in $[n-2]\setminus S$.
\end{itemize}   

In order to realize the pants decomposition corresponding to $G$, it is enough to choose a curve from each homotopy class of curves in $\Phi_G$ such that these curves are pairwise disjoint.

Let $\Lambda=\bigcup_{S\subset [n-2]} (\Gamma_S\cup \Theta_S)$. 
As in the proof of \cref{lem:labelled sphere}, we first prove that the set of homotopy classes in $\Lambda$ are enough to recognize every graph in $\G$. 

 \textbf{Claim.} The set of homotopy classes $\Lambda$ recognizes all graphs in $\G$.
 
For every graph $G\in\G$, we proceed as follows. We fix a vertex $v$ in the cycle and let $e=vw$ be one of its adjacent edges that belongs to the cycle.
The edge $e$ is non-separating and we realize this edge for every graph $G\in\G$ by the curve $\Theta_\emptyset=\theta$, see Figure~\ref{Fig:nonseparating}.

Let $T$ be the tree that we obtain by removing $e$ from $G$ and let $v$ be the root of $T$. 
Note that $T$ is almost trivalent; only the vertices incident to $e$ ($v$ and $w$) are of degree two. We recognize the edges in $T$ by the homotopy classes in $\bigcup_{S\subset [n-2]} \Gamma_S$ as in the proof of \cref{lem:labelled sphere} and with the following additional considerations:
\begin{itemize}
\item In the proof of \cref{lem:labelled sphere}, during the induction step where one removes a vertex to obtain three subtrees, an arbitrary choice was made to select that the curves in one subtree were going above the other punctures, the curve in another subtree were going below and the curves in the last subtree were between. Here, this choice is no longer arbitrary, and we enforce the fact that for each edge in the cycle of $G$ except $e$, the choice is made so that the corresponding curve goes \emph{above} the remaining punctures. This is pictured in Figure~\ref{Fig:genus1}.
\item If $w$ is adjacent to a labeled vertex $i$, the edge $iw$ is not internal and does not need to be recognized. We recognize the other edge adjacent to $w$ by a homotopy class in $\Gamma_{\{i\}}$ (note that in the proof of \cref{lem:labelled sphere} we did not need the homotopy classes of curves that separated a single puncture). Otherwise both edges adjacent to $w$ are recognized with the same homotopy class as the case where these two edges are replaced by one edge. 
\end{itemize}

Now, denote by $\phi_r$ the homotopy class of curves associated to the edge $r$ in $T$. We recognize the edges in $G$ as follows. If $r$ is an edge in $G$ that does not belong to the cycle in $G$, we recognize it by $\phi_r$; otherwise, we recognize it by the corresponding homotopy class in $\Theta_S$ that is the concatenation of $\phi_r$ and $\theta$. Denote these curves together with $\Theta_\emptyset=\theta$ by $\Phi_G$. We still need to prove that these homotopy classes are disjoint and that they indeed recognize the edges in $G$.

By construction, the curves $\phi_r$ are pairwise disjoint, and so are their counterparts in $\Theta_S$. The remaining disjointnesses follows from the first additional consideration above: the curves in $\phi_r$ when $r$ belongs to the cycle all go above the punctures that they did not encompass, and thus they can be safely concatenated with $\theta$ while still being disjoint from all the other curves in $\Phi_G$.

Finally, recognizing the edges not in the cycle follows the same argument as in the case without genus. For pairs of edges in the cycle, the union of the corresponding pair of curves forms a homology cycle that separates the corresponding punctures as desired. This finishes the proof of the claim. 
$\hfill\blacksquare$

To conclude, it now suffices to pick a curve in each homotopy class of $\Lambda$, such that they are altogether in minimal position. 
For each $G$, the curves originating from $\Phi_G$ are pairwise disjoint and therefore realize the pants decomposition corresponding to $G$. Therefore, the total family of curves realizes all types of pants decompositions.

\begin{figure}[t]
    \centering
    \includegraphics[width=.8\textwidth]{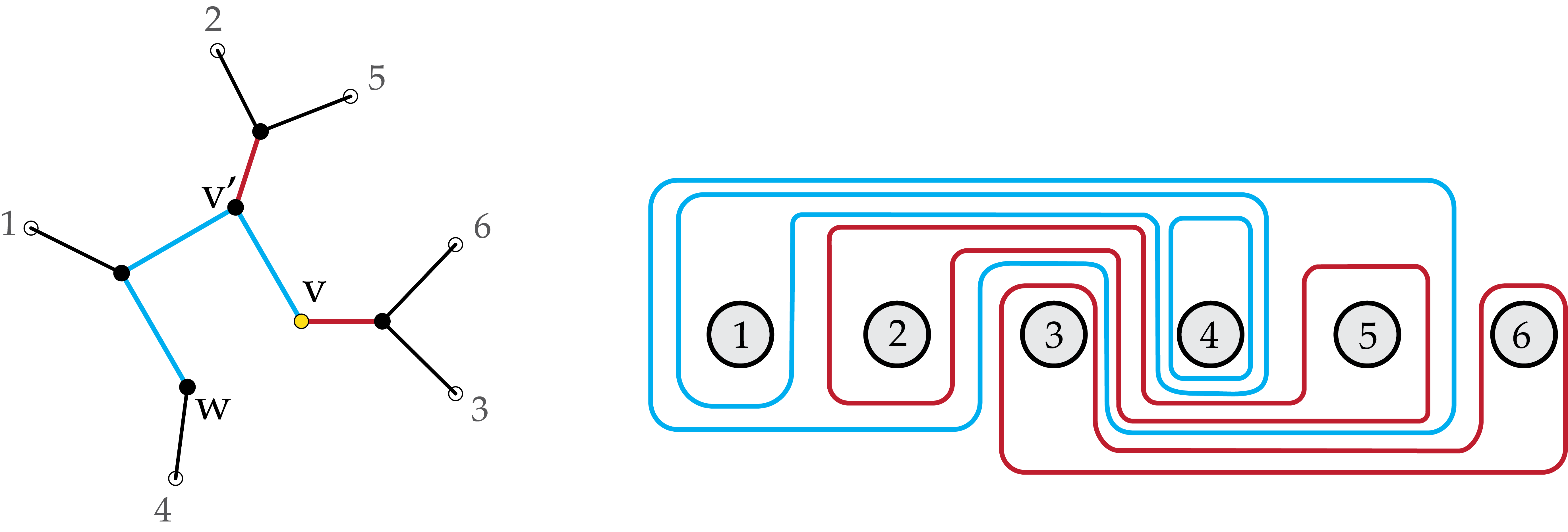}
    \caption{The edges in the cycle of the graph are recognized by curves that go above the remaining punctures in (the planarization of) $\Sigma$.}
    \label{Fig:genus1}
\end{figure}

\begin{figure}[htp]
    \centering
    \includegraphics[width=.85\textwidth]{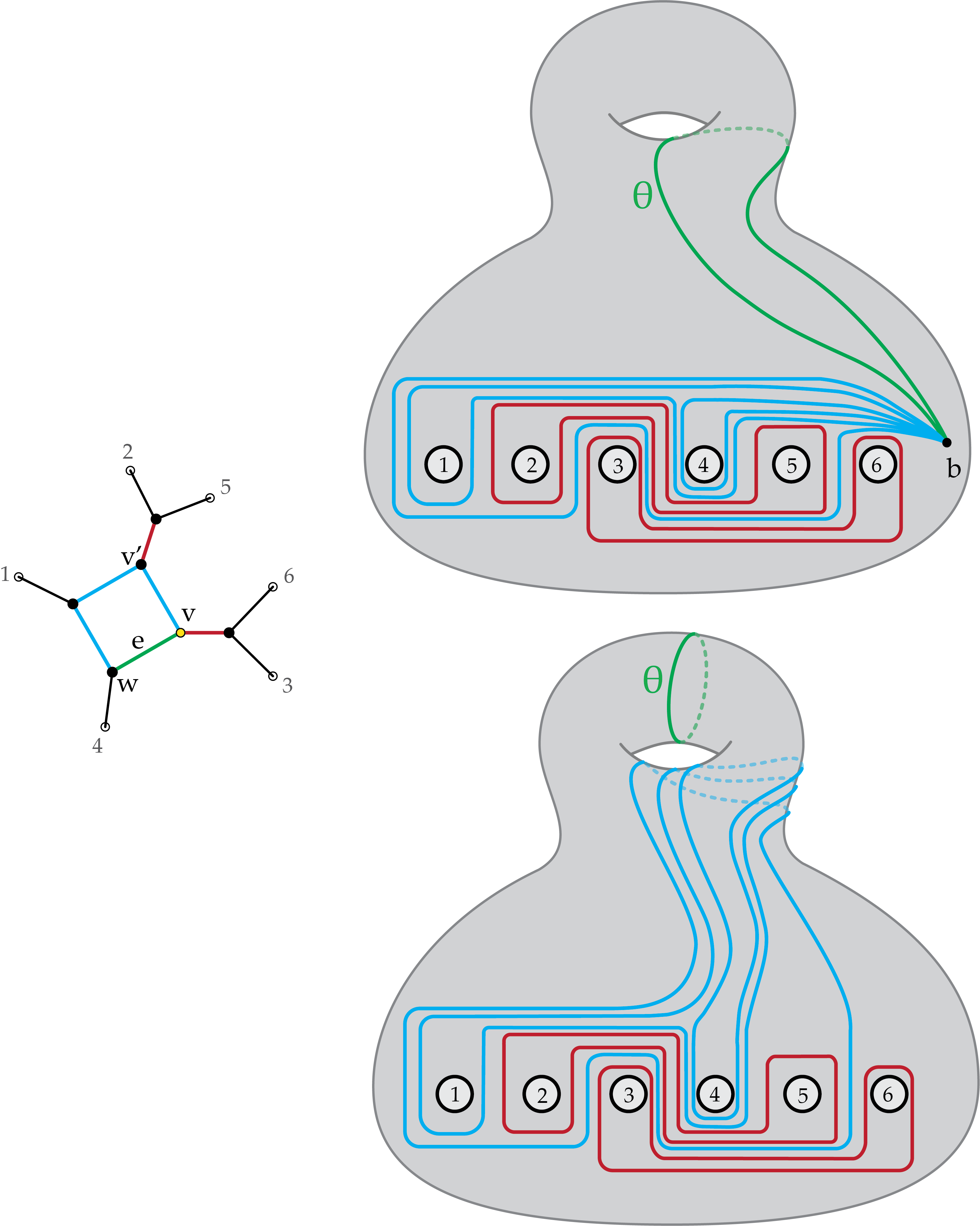}
    \caption{Realizing $G$ (at left) from a realization for $T$. The blue curves correspond to the edges that belong to the cycle in $G$ and the red ones correspond to those that do not. The homotopy classes of the curves in the right bottom drawing recognizes the graph $G$.}
    \label{Fig:concatenation}
\end{figure}

Finally, we provide an upper bound the size of the family $\Lambda$. The family $\Theta_S$ consists of only one curve for each choice of $S$, and thus, reusing the bound from Lemma~\ref{lem:labelled sphere}, we obtain
\[|\Lambda|=\Sum_{S\subset [n]}(|\Gamma_S|+ |\Theta_S|)=2^{n-2}+\frac{1}{4}(3^{n-2}-2n+3)<3^{n-2}.\]
\end{proof}

We next look at the case of genus $2$ with $n-4$ punctures, where we can do even better than in the previous case.

\begin{lemma}
There exists a family of simple closed curves of size at most $3^{n-3}$ that realizes all types of pants decompositions of $\Sigma_{2,n-4}$ up to labelled homeomorphisms.
\end{lemma}

\begin{proof}
The dual graph to a pants decomposition of $\Sigma_{2,n-4}$ is a graph with cyclomatic number two (i.e., there are two edges so that removing them yields a tree) in which all vertices have degree three except $n-4$ vertices of degree one; these vertices correspond to the labelled punctures and are labelled. Let us denote all such graphs with $\G$. These graphs have either two disjoint cycles or two cycles that share at least one or more edges, see Figure~\ref{Fig:duals}. Note that in this case the edges that belong to the cycles are non-separating and therefore correspond to non-separating closed curves on the surface. Also, if the cycles are disjoint, the edges that belong to the path connecting the cycles are separating the graph into two sub-graphs with one cycle each and therefore correspond to separating closed curves that cut the surface into two genus one surfaces.

\begin{figure}[H]
    \centering
    \includegraphics[width=.55\textwidth]{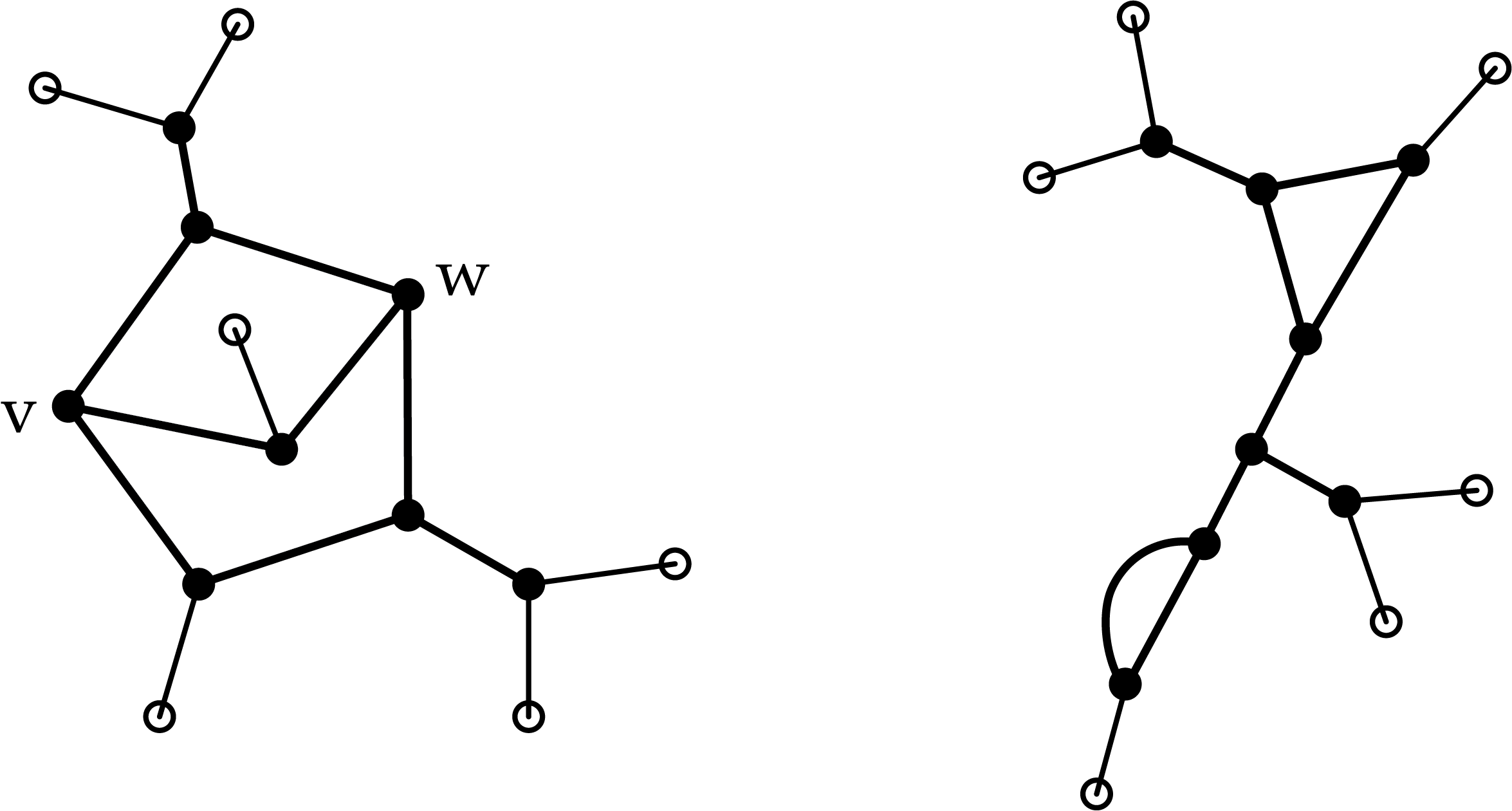}
    \caption{Left: the cycles are not disjoint; we can see that there are two vertices with all adjacent edges non-separating. Right: two disjoint cycles.}
    \label{Fig:duals}
\end{figure}
Fix a base point $b$ which we intuitively consider next to the puncture $n-4$, as depicted in Figure~\ref{Fig:omega}. We consider a family of simple closed curves $\theta_1,\theta_2,\theta_3,\psi, \omega$ based at $b$ as pictured in Figure~\ref{Fig:omega}: the curves $\theta_i$ are non-separating and pairwise disjoint, the curve $\psi$ is also non-separating and only crosses $\theta_1$ once, and the curve $\omega$ is separating the handle formed by $\theta_1$ and $\psi$ from the rest of the surface.

 Denote by $\Sigma'$ the surface we obtain by cutting along a curve in $\theta_1$ and a curve in $\theta_2$. The surface $\Sigma'$ is a sphere with $n$ punctures among which $n-4$ are labelled. For a subset $S$ of the labelled  punctures, let $\Gamma'_S$ be the homotopy classes of separating curves that encompass the punctures in $S$ and remain either above or below the rest of the punctures as defined in the proof of \cref{lem:labelled sphere}. Let $\Gamma_S$ be the homotopy classes of closed curves on $\Sigma_{2,n-4}$ that we obtain from $\Gamma'_S$ by gluing back the surface along $\theta_1$ and $\theta_2$. Note that by construction, $\theta_i$ for $1\leq i\leq 3$ and $\omega$ are disjoint from $\Gamma_S$ for $S\subset[n-4]$.

Define $\Omega_S$ to be the concatenation of curves in $\Gamma_S$ with $\omega$; these curves are separating. Finally, define $\Theta^i_S$ to be the concatenation of curves in $\Gamma_S$ with the $\theta_i$ for $i=1,2,3$; these curves are not separating. 

\begin{figure}[H]
    \centering
    \includegraphics[width=.95\textwidth]{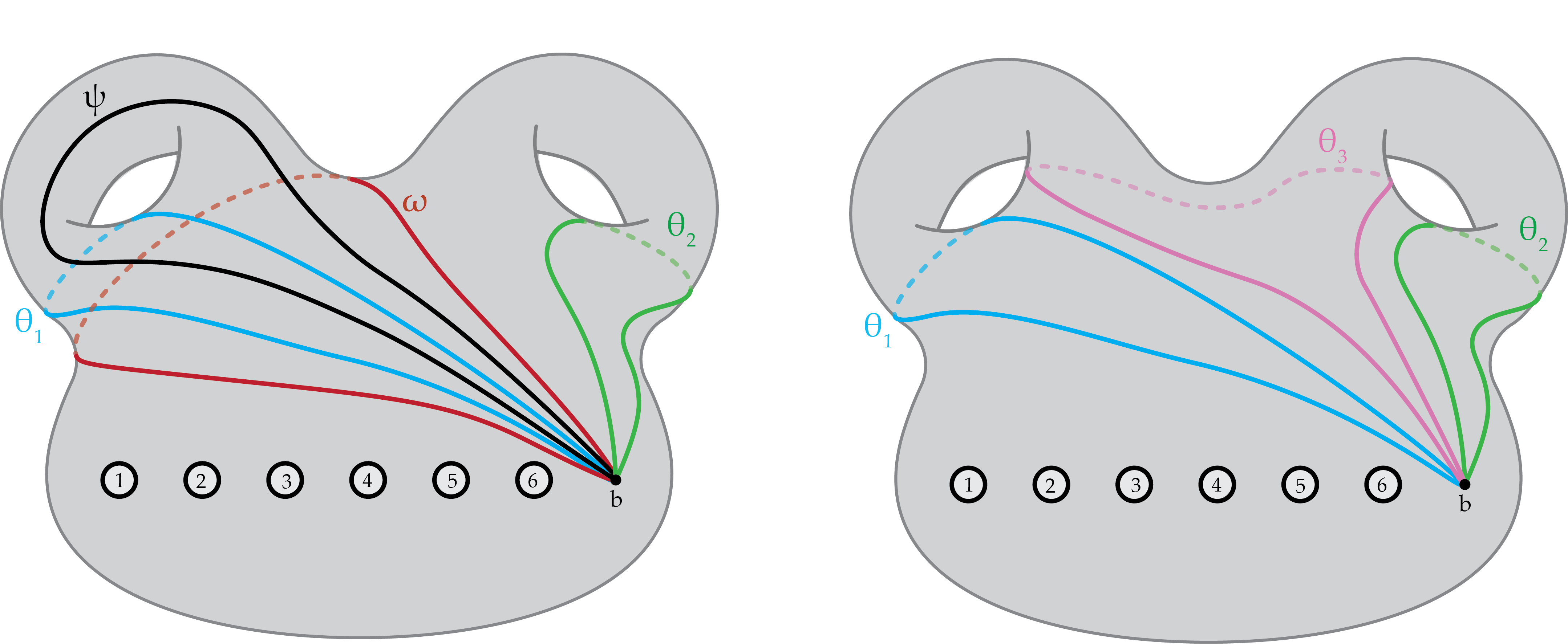}
    \caption{Homotopy classes of the non-separating curves $\theta_1,\theta_2,\theta_3$ and $\psi$ and the separating curve $\omega$.}
    \label{Fig:omega}
\end{figure}

In the proof of Lemma~\ref{lem:labelled torus}, we introduced a notion for a set of curves recognizing a graph. Here we provide a more intricate variant of this notion in the genus $2$ setting. We say that a set $\Phi_G$ of based homotopy classes of curves recognizes a graph $G\in \G$, if it satisfies the following properties.
\begin{itemize}
    \item The free homotopy classes in $\Phi_G$ are pairwise disjoint.
\item For each internal edge $e$ in $G$ that is separating a tree that contains the subset $S$ of the labelled vertices, there exists a homotopy class of curves in $\Phi_G$ that cuts the surface into a sphere containing the punctures in $S$ and a genus two surface containing the punctures in $[n-4]\setminus S$.
\item Let $e$ and $e'$ be a pair of edges that separate a tree from $G$ containing the labelled vertices in $S$. 
Then there exist two homotopy classes of non-separating curves in $\Phi_G$ corresponding to $e$ and $e'$ that together separate the surface into a sphere that contains the punctures in $S$ and a genus one surface containing the punctures in $[n-4]\setminus S$.
\item If an edge separates the graph into two sub-graphs with one cycle such that one contains the vertices in $S$ and the other contains the vertices in $[n-4]\setminus S$, then there exists a homotopy class in $\Phi_G$ that separates the surface into two surfaces of genus one such that one contains the punctures in $S$ and the other contains the punctures in $[n-4]\setminus S$.
\end{itemize}

In order to realize the pants decomposition corresponding to $G$, it is enough to choose a curve from each homotopy class of curves in $\Phi_G$ such that these curves are pairwise disjoint.

Let $\Lambda= \bigcup_{S\subset [n-4]} (\Gamma_S\cup \Theta^1_S\cup\Theta^2_S\cup\Theta^3_S\cup \Omega_S)$. We first prove that the homotopy classes in $\Lambda$ are enough to recognize all graphs in $\G$. Then we show that we can choose a curve from each homotopy class in $\Lambda$ such that these curves realize all types of pants decompositions of $\Sigma_{2,n-4}$.

 \textbf{Claim.} The set of homotopy classes $\Lambda$ is enough to recognize all graphs in $\G$.
 
We consider two different cases for the proof of this claim.

\textbf{The case where the cycles in $G$ are not disjoint.}
In this case, there exist exactly two vertices $v$ and $w$ such that all their adjacent edges are non-separating and belong to the cycles, see left picture in Figure~\ref{Fig:duals}. By removing one of these vertices from $G$, let us say $w$, we obtain a tree; denote this tree by $T$. Consider the vertex $v$ to be the root of $T$ and proceed as in \cref{lem:labelled sphere} to recognize the edges of $T$ with homotopy classes in $\bigcup_{S\subset [n-4]} \Gamma_S$. Number the branches of $T$ at $v$ by 1,2 and 3 (there might only exist 2 branches or even one). As in \cref{lem:labelled torus}, we have the power to choose the homotopy classes corresponding to the edges that belong to the cycles and are in branch 1 to go above everything else. Likewise, we ensure that the cycles corresponding to branch $2$ go below everything else, and finally that the edges corresponding to branch $3$ go inbetween the cycles of branches $1$ and $2$. Let $\Phi_T$ be the set of homotopy classes of curves that recognize $T$. We denote the curves in $\Phi_T$ that recognize the edge $r$ in $T$ by $\phi_r$.

In order to recognize $G$, we proceed as follows. If $r$ is an edge in $G$ that does not belong to any cycle in $G$, we recognize it by $\phi_r$; otherwise we recognize it by a homotopy class of curves that is the concatenation of $\phi_r$ with $\theta_i$ if it belongs to the branch $i$ in $T$ for $1\leq i\leq 3$. We recognize the edges adjacent to $w$ by the curves $\Theta^1_\emptyset=\theta_1, \Theta^2_\emptyset=\theta_2$ and $\Theta^3_\emptyset=\theta_3$. Denote these homotopy classes of curves by $\Phi_G$. Now, these free homotopy classes can be realized by disjoint curves as the choices of above/below in the three branches have been specifically designed so that concatenating those curves with $\theta_1, \theta_2$ and $\theta_3$ can still be done while preserving disjointedness: see Figure~\ref{Fig:p2}.

\begin{figure}[ht]
    \centering
    \includegraphics[width=.9\textwidth]{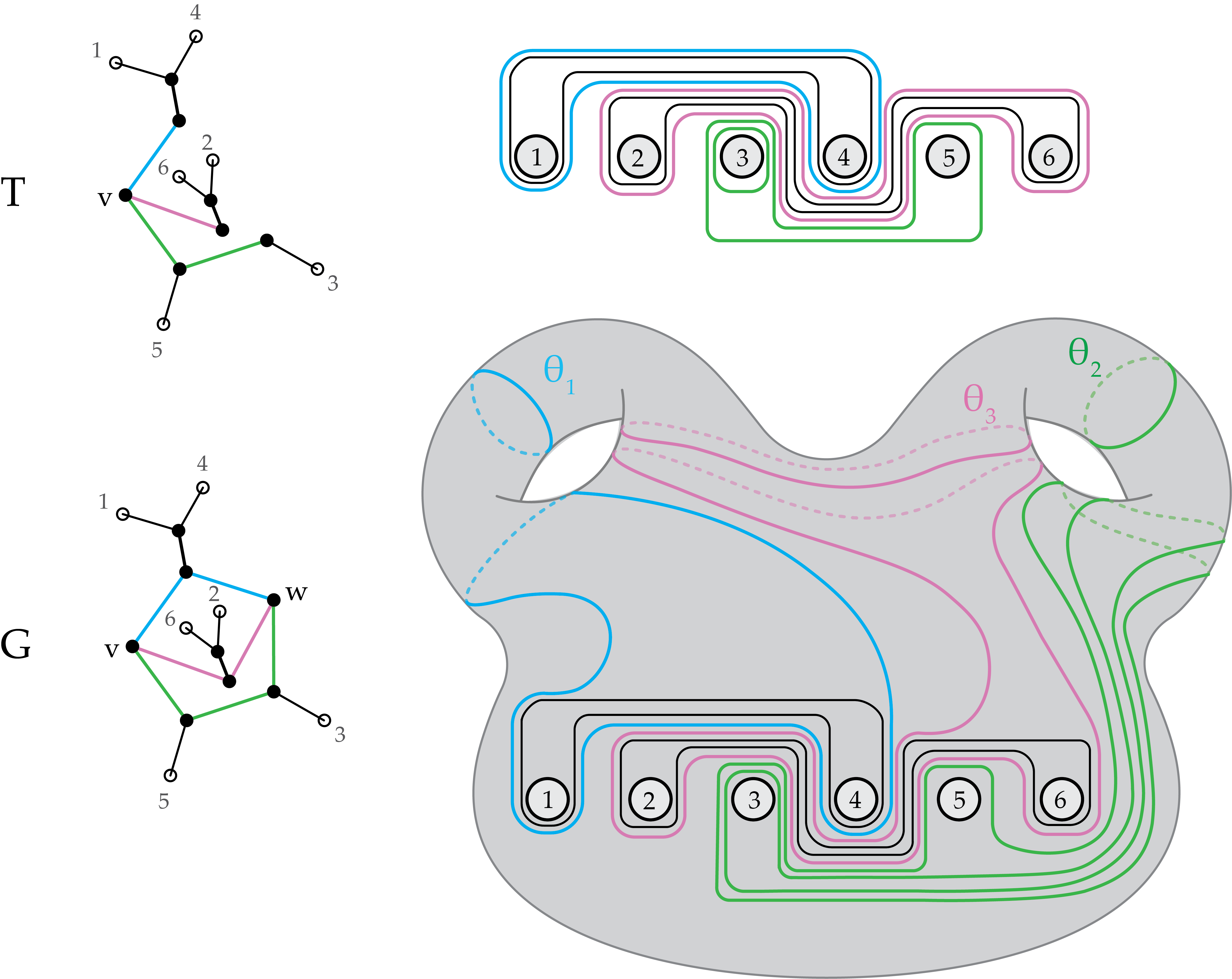}
    \caption{The case where the cycles in $G$ are not disjoint.}
    \label{Fig:p2}
\end{figure}
\begin{figure}[ht]
    \centering
    \includegraphics[width=.9\textwidth]{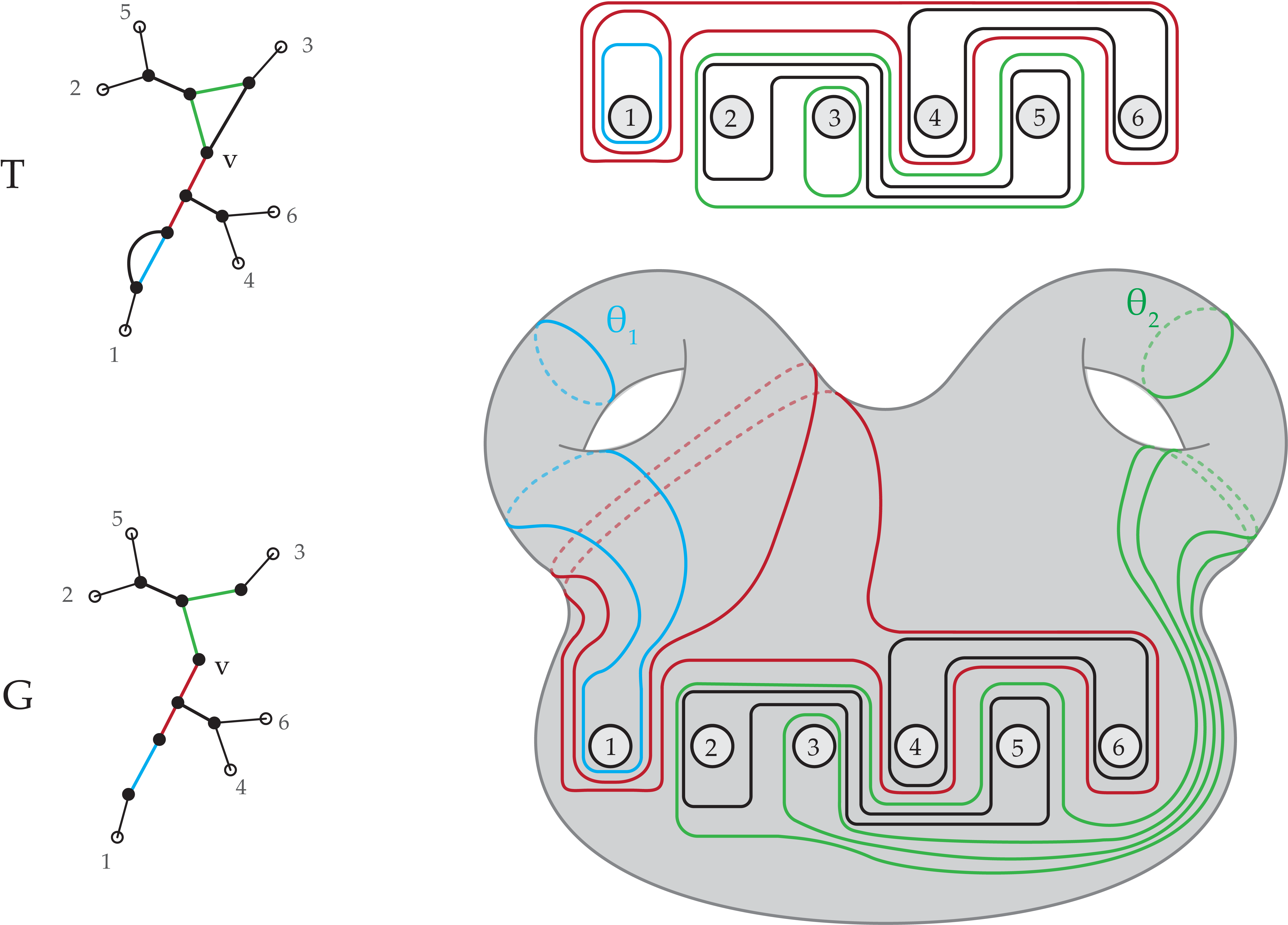}
    \caption{The case where the cycles in $G$ are disjoint.}
    \label{Fig:p3}
\end{figure}

\textbf{The case where $G$ has two disjoint cycles.}
Choose the vertex $v$ in a cycle of $G$ such that it belongs to the path that connects the two cycles. Remove an edge $r_1$ from the other cycle and an edge $r_2$ adjacent to $v$ from the cycle that $v$ belongs to. We recognize $r_1$ by $\theta_1$ and $r_2$ by $\theta_2$ in Figure~\ref{Fig:omega}. The graph $G\setminus \{r_1,r_2\}$ is a tree which we denote by $T$. Let $v$ be the root of $T$ and proceed as in \cref{lem:labelled sphere} to realize $T$. We choose the homotopy classes that recognize the edges in the path that connect the two cycles and the edges in the cycle that contains $r_1$ such that they go above the other punctures. 
For the edges in the cycle that contains $r-2$, we choose homotopy classes that go below the rest of the punctures.
Let $\Phi_T$ be the set of homotopy classes of curves that recognize $T$. We denote the homotopy class in $\Phi_T$ that recognizes the edge $r$ in $T$, by $\phi_r$.

In order to recognize $G$, we proceed as follows. If $r$ is an edge in $G$ that does not belong to any cycle in $G$ nor the path connecting the two cycles, we realize it by $\phi_r$. If $r$ belongs to the same cycle as $v$ in $G$, we concatenate $\phi_r$ with $\theta_2$ and if it belongs to the other cycle, we concatenate it with $\theta_1$. If $r$ belongs to the path connecting the two cycles, $r$ is recognized by the concatenation of $\phi_r$ with $\omega$. These curves together with $\Theta^1_\emptyset=\theta_1$ and $\Theta^2_\emptyset=\theta_2$ recognize $G$; we denote the set containing these curves by $\Phi_G$. As before, the choices of above and below in the homotopy classes have been specifically designed to ensure that the homotopy classes can be realized with disjoint curves, see Figure~\ref{Fig:p3}. 

This finishes the proof of the claim. $\hfill\blacksquare$

As in the proofs of the previous lemmas, it now suffices to pick a closed curve in each homotopy class and $\Lambda$ so that the resulting is in minimal position. By construction, for any type of pants decomposition with dual graph $G$, the curves in $\Phi_G$ are pairwise disjoint and realize this type of pants decomposition. Altogether, the curves that we chose realize all types of pants decompositions.

Finally, we provide an upper bound on the size of $\Lambda$. Note that in our construction, it suffices to consider curves in $\Theta^1_S$ and $\Omega_S$ which go above the punctures they do not encompass, and curves for $\Theta^2_S$ which go below the punctures they do not encompass. However, for $\Theta^3_S$, as for $\Gamma_S$, we need to consider all possible choices of above and below for the punctures they do not encompass. This yields the follow bound.

\begin{align*}&&|\Lambda|=& \  \Sum_{S\subset [n]}(|\Theta^1_S|+|\Theta^2_S|+|\Omega_S|)+\Sum_{S\subset [n]} (|\Gamma_S|+|\Theta^3_S|)\\&&=& \ 3 \times 2^{n-4}+2\times\frac{1}{4}(3^{n-4}-2(n-4)-1)\\&&<& \  3^{n-3}\end{align*}

\end{proof}

We do not push these studies farther: it would seem that the techniques used in this last proof would require considering curves using arbitrary subsets of the handles, and thus naturally lead to exponentially-sized family of curves. While the resulting base of the exponential could conceivably be smaller than the one we obtain in Lemma~\ref{lem:allpants}, it would therefore not be helpful towards breaking the exponential barrier.

{\bf Acknowledgements.} We are grateful to Bram Petri for pointing out a mistake in an earlier version of this paper.

\bibliographystyle{alpha}
\bibliography{biblio}

\newcommand{\etalchar}[1]{$^{#1}$}
\begin{thebibliography}{HKDMT17}

\bibitem[Alo17]{alon}
Noga Alon.
\newblock Asymptotically optimal induced universal graphs.
\newblock {\em Geometric and Functional Analysis}, 27:1--32, 2017.

\bibitem[BGP20]{bonamy}
Marthe Bonamy, Cyril Gavoille, and Micha{\l} Pilipczuk.
\newblock Shorter labeling schemes for planar graphs.
\newblock In {\em Proceedings of the Fourteenth Annual ACM-SIAM Symposium on
  Discrete Algorithms}, pages 446--462. SIAM, 2020.

\bibitem[Bol82]{bollobas}
B{\'e}la Bollob{\'a}s.
\newblock The asymptotic number of unlabelled regular graphs.
\newblock {\em Journal of the London Mathematical Society}, 2(2):201--206,
  1982.

\bibitem[Bus92]{buser}
Peter Buser.
\newblock {\em Geometry and spectra of compact {R}iemann surfaces}, volume 106
  of {\em Progress in Mathematics}.
\newblock Birkh\"{a}user Boston, Inc., Boston, MA, 1992.

\bibitem[CP12]{10.1215/00127094-1507312}
William Cavendish and Hugo Parlier.
\newblock {Growth of the Weil–Petersson diameter of moduli space}.
\newblock {\em Duke Mathematical Journal}, 161(1):139 -- 171, 2012.

\bibitem[DEG{\etalchar{+}}21]{dujmovic}
Vida Dujmovi{\'c}, Louis Esperet, Cyril Gavoille, Gwena{\"e}l Joret, Piotr
  Micek, and Pat Morin.
\newblock Adjacency labelling for planar graphs (and beyond).
\newblock {\em Journal of the ACM (JACM)}, 68(6):1--33, 2021.

\bibitem[EJM20]{esperet}
Louis Esperet, Gwena{\"e}l Joret, and Pat Morin.
\newblock Sparse universal graphs for planarity.
\newblock {\em arXiv preprint arXiv:2010.05779}, 2020.

\bibitem[F{\'a}r48]{fary}
Istv{\'a}n F{\'a}ry.
\newblock On straight line representations of planar graphs.
\newblock {\em Acta scientiarum mathematicarum (Szeged)}, 11:229--233, 1948.

\bibitem[FM11]{farb2011primer}
Benson Farb and Dan Margalit.
\newblock {\em A primer on mapping class groups (pms-49)}.
\newblock Princeton university press, 2011.

\bibitem[GPY11]{guth2011pants}
Larry Guth, Hugo Parlier, and Robert Young.
\newblock Pants decompositions of random surfaces.
\newblock {\em Geometric and Functional Analysis}, 21(5):1069, 2011.

\bibitem[Gre18]{greene}
Joshua~Evan Greene.
\newblock On curves intersecting at most once, ii.
\newblock {\em arXiv preprint arXiv:1811.01413}, 2018.

\bibitem[HKDMT17]{shortestpaths}
Alfredo Hubard, Vojt{\v{e}}ch Kalu{\v{z}}a, Arnaud De~Mesmay, and Martin
  Tancer.
\newblock Shortest path embeddings of graphs on surfaces.
\newblock {\em Discrete \& Computational Geometry}, 58(4):921--945, 2017.

\bibitem[Mil09]{miller}
Alison Miller.
\newblock Asymptotic bounds for permutations containing many different
  patterns.
\newblock {\em Journal of Combinatorial Theory, Series A}, 116(1):92--108,
  2009.

\bibitem[Mir13]{mirzakhani}
Maryam Mirzakhani.
\newblock Growth of {W}eil-{P}etersson volumes and random hyperbolic surfaces
  of large genus.
\newblock {\em J. Differential Geom.}, 94(2):267--300, 2013.

\bibitem[Prz15]{przytycki}
Piotr Przytycki.
\newblock Arcs intersecting at most once.
\newblock {\em Geometric and Functional Analysis}, 25(2):658--670, 2015.

\bibitem[Rad64]{rado}
R.~Rado.
\newblock Universal graphs and universal functions.
\newblock {\em Acta Arithmetica}, 9(4):331--340, 1964.

\bibitem[Riv02]{rivin}
Igor Rivin.
\newblock Counting cycles and finite dimensional lp norms.
\newblock {\em Advances in applied mathematics}, 29(4):647--662, 2002.

\bibitem[Wag36]{wagner}
K.~Wagner.
\newblock Bemerkungen zum {V}ierfarbenproblem.
\newblock {\em Jahresbericht der {D}eutschen {M}athematiker-{V}ereinigung},
  46:26--32, 1936.

\bibitem[Wu19]{wu2019growth}
Yunhui Wu.
\newblock Growth of the weil--petersson inradius of moduli space.
\newblock In {\em Annales de l'Institut Fourier}, volume~69, pages 1309--1346,
  2019.

\end{thebibliography}
{\em Addresses:}\\
LIGM, CNRS, Univ. Gustave Eiffel, ESIEE Paris, F-77454 Marne-la-Vallée, France\\
{\em Emails:} \href{niloufar.fuladi@univ-eiffel.fr}{niloufar.fuladi@univ-eiffel.fr}, \href{mailto:arnaud.de-mesmay@univ-eiffel.fr}{arnaud.de-mesmay@univ-eiffel.fr}\\
Department of Mathematics, University of Luxembourg, Luxembourg \\
{\em Emails:} \href{mailto:hugo.parlier@unilu.ch}{hugo.parlier@uni.lu}\\
\end{document}